\documentclass[10pt]{article}
\usepackage{latexsym,amsfonts,amssymb,amsthm,amsmath,mathrsfs,color,amscd,graphicx,fullpage,fancyhdr,parskip,hyperref}

\newcommand{\s}[1]{{\mathcal #1}}
\newcommand{\sr}[1]{{\mathscr #1}}
\newcommand{\bb}[1]{{\mathbb #1}}

\newtheorem{theorem}{Theorem}[section]
\newtheorem{corollary}[theorem]{Corollary}
\newtheorem{conjecture}[theorem]{Conjecture}
\newtheorem{lemma}[theorem]{Lemma}
\newtheorem{proposition}[theorem]{Proposition}
\newtheorem{problem}[theorem]{Problem}
\newtheorem{definition}[theorem]{Definition}

\newtheorem{remark}[theorem]{Remark}

\numberwithin{equation}{section}
\numberwithin{theorem}{section}

\title{Optimal control of first-order Hamilton-Jacobi equations with linearly bounded Hamiltonian}
\author{P. Jameson Graber\footnote{Commands team (ENSTA ParisTech, INRIA Saclay), 828, Boulevard des Maréchaux, 91762 Palaiseau Cedex, Email: philip.graber@inria.fr}}

\begin{document}

\maketitle

\begin{abstract}
We consider the optimal control of solutions of first order Hamilton-Jacobi equations, where the Hamiltonian is convex with linear growth. This models the problem of steering the propagation of a front by constructing an obstacle. We prove existence of minimizers to this optimization problem as in a relaxed setting and characterize the minimizers as weak solutions to a mean field game type system of coupled partial differential equations. Furthermore, we prove existence and partial uniqueness of weak solutions to the PDE system. An interpretation in terms of mean field games is also discussed.

\end{abstract}

\tableofcontents

\section{Introduction} \label{sec:introduction}

Let $\Omega_t \subset \bb{R}^N$ be an evolving set with $\Gamma_t := \partial \Omega_t$ modeling a front propagating in space. A simple model of front propagation is in terms of reachable sets for an open-loop controlled dynamical system 
\begin{equation} \label{eq:open_loop}
\dot{y}(\tau) = c(y(\tau),\alpha(\tau)), \tau \in \bb{R},
\end{equation}
where the control $\alpha(\tau)$ takes values in a compact control space $A$. Consider times $t$ in a finite time horizon $[0,T]$ and a given closed set $\Omega_T$. We consider $\Omega_t$ to be the {\em backward reachable set} at time $t$ defined by
\begin{equation}
\Omega_t := \{y(t) : y ~\text{solves}~ (\ref{eq:open_loop}), ~ \alpha \in L^\infty(t,T;A), y(T) \in \Omega_T\}.
\end{equation}
Let $u_T$ be a continuous function such that the set where $u_T \leq 0$ coincides precisely with $\Omega_T$. Then we see that $\Omega_t$ given above coincides precisely with the set
\begin{equation}
\{y(t) :  u_T(y(T)) \leq 0, ~ y ~\text{solves}~ (\ref{eq:open_loop}), ~ \alpha \in L^\infty(t,T;A) \}.
\end{equation}
In other words, $\Omega_t$ coincides with the sub-level set $\{u(t,x) \leq 0\}$, where $u$ is the value function of the simple optimal control problem
\begin{equation} \label{eq:mayer}
u(t,x) := \inf \left\{u_T(y(T)) : y ~\text{solves}~ (\ref{eq:open_loop}), ~ \alpha \in L^\infty(t,T;A), y(t) = x \right\}.
\end{equation}
This approach to modeling front propagation is known as the level-set method. In particular, if $u_T < 0$ in the interior of $\Omega_t$, then then front $\Gamma_t = \partial \Omega_t$ coincides exactly with the level set $\{u(t,x) = 0\}$. One feature of this approach is that, under classical assumptions, $u$ is a continuous viscosity solution to the following Hamilton-Jacobi equation.
\begin{equation} \label{eq:eikonal}
-u_t(t,x) + H(x,Du(t,x)) = 0, ~~~ u(T,x) = u_T(x),
\end{equation}
where the {\em Hamiltonian} is given by
\begin{equation} \label{eq:hamiltonian}
H(x,p) := \sup \{-c(x,a)\cdot p : a \in A \}.
\end{equation}
Thus solving the PDE in a viscosity sense and then tracking level sets can be said to solve the problem of observing front propagation. For more on the theory of optimal control and Hamilton-Jacobi equations, see \cite{bardi97} and references therein.

In this study, we consider the question of {\em controlling} front propagation. One approach to this problem has been studied in several papers by Bressan and others, e.g. \cite{bressan07,bressan08,bressan09,bressan09a,bressan09b}. In this approach one defines an obstacle or ``blocking strategy" $\gamma = \gamma(t)$, a continuous curve constructed in real time. Then the new reachable set is given (in this instance) by 
\begin{equation} \label{eq:dynamic_blocking}
\Omega_t^\gamma = \{y(t) : y ~\text{solves}~ (\ref{eq:open_loop}), ~ \alpha \in L^\infty(t,T;A), y(T) \in \Omega_T, y(\tau) \notin \gamma(\tau) ~ \forall ~ \tau \in [t,T]\}.
\end{equation}
One can then formulate an optimal control problem where the functional to be minimized is the construction cost of admissible blocking strategies. There are several open problems along this line of study \cite{bressan13}, but we will not consider them here. Note that in this framework the obstacles are constructed as impenetrable curves, rather than more general obstacles which may take up space. One disadvantage of this approach is a lack of a natural characterization in terms of a PDE (see, however, the work of De Lellis and Robyr \cite{delellis11}).

Here we consider another approach. Suppose we model an obstacle by a non-negative continuous function $f = f(t,x)$. We define the new reachable set to be
\begin{equation}
\tilde{\Omega}_t = \{y(t) : u_T(y(T)) + \int_t^T f(s,y(s)) \leq 0, ~ y ~\text{solves}~ (\ref{eq:open_loop}), ~ \alpha \in L^\infty(t,T;A), y(T) \in \Omega_T\}.
\end{equation}
In other words, $\tilde{\Omega}_t$ is the set of all points $x$ in the reachable set $\Omega_t$ for which the trajectory taken to reach $x$ is not too costly, where the cost is determined by $f$. It follows that $\tilde{\Omega}_t$ is precisely the sub-level set $\{\tilde{u} \leq 0\}$ where $\tilde{u}$ is given by the control problem
\begin{equation} \label{eq:value_function}
\tilde{u}(t,x) := \inf \left\{u_T(y(T)) + \int_t^T f(s,y(s)) ds : y ~\text{solves}~ (\ref{eq:open_loop}), ~ \alpha \in L^\infty(t,T;A), y(t) = x \right\}.
\end{equation}
We can again write $\tilde{u}$ as a solution to a Hamilton-Jacobi equation
\begin{equation} \label{eq:hjb}
-u_t(t,x) + H(x,Du(t,x)) = f(t,x), ~~~ u(T,x) = u_T(x).
\end{equation}
With this approach, total blocking of the front is possible under certain assumptions on the data, but for general purposes it is not necessary to think in terms of blocking (for more discussion, see Section \ref{sec:conclusion}). In any case, one can hope to steer the reachable set using $f$ as a control so that it has a minimal presence in a specified region. Suppose we are given a measure $m_0$. We can think of the sets on which $m_0$ is most concentrated (i.e. those sets with large measure) as ``danger zones," regions where we would rather the front not reach. As a general rule, the functional to be maximized is $\int \tilde{u}(T,x) dm_0(x).$
At the same time, we would like to minimize the cost of constructing $f$. We arrive at the optimization problem
\begin{problem} \label{pr:main_problem}
For a given cost functional $K$ and finite measure $dm_0$, find
\begin{equation} \label{eq:main_problem}
\inf \left\{ \int_0^T\int K(t,x,f(t,x)) dx dt - \int u(0,x)dm_0(x) \right\}
\end{equation}
over functions $u,f$ such that $u$ solves (\ref{eq:hjb}).
\end{problem}
In this paper, we will focus on the questions of existence, uniqueness, regularity, and characterization of the minimizers as solutions to a system of partial differential equations which, as we will see, has independent interest. Our main results are Theorem \ref{thm:minimizers_weak}, which characterizes minimizers of (a version of) Problem \ref{pr:main_problem} as weak solutions to a mean field games type system (\ref{eq:mfg}) and, combined with Theorem \ref{thm:existence}, proves existence and almost uniqueness of these solutions.

Problems of this type have been studied before in the very different context of Mean Field Games \cite{lasry06,lasry06a,lasry07}. A recent article by Cardaliaguet \cite{cardaliaguet13}, see also \cite{cardaliaguet12a}, gives the closest example to our present study and inspires many of our results, but it nevertheless has some crucial differences in terms of the Hamiltonian structure. The most obvious difference is the order of the bounds on the Hamiltonian, which in \cite{cardaliaguet13} are of the form $C|Du|^p + C_1$ for $p > 1$; here we take $p = 1$. In fact, this difference removes some of the strict convexity in the problem and destroys certain regularity properties of solutions, such as continuity of the solution and integrability of the derivative. Another difference is in terms of the variable coefficients. The Hamiltonian $H(x,Du) = c(x)|Du|^p$ is explicitly excluded in \cite{cardaliaguet13} because of the convolution method used in order to obtain smooth approximations of weak solutions to the Hamilton-Jacobi equation. In this paper we bypass this problem by avoiding this sort of regularization argument, and instead we use in a critical way the role of duality of the continuity equation. See Section \ref{sec:relaxed} below for a detailed explanation.

The paper is organized as follows. In Section \ref{sec:existence_regularity}, we establish the mathematical setting of our problem, list the main assumptions, and then proceed to analyze the question of existence of minimizers. The main question which we address is how to construct an appropriate (nonsmooth) function space in which minimizers of the original problem may be said to exist. Section \ref{sec:characterization} is devoted to the main result, Theorem \ref{thm:minimizers_weak}, which is a characterization of the minimizers as weak solutions to a system of partial differential equations (in fact, a mean field game type system). Finally, in Section \ref{sec:conclusion} we give some applications for further consideration.

\section{Existence of minimizers and regularity}

\label{sec:existence_regularity}

We consider the backward Hamilton-Jacobi equation
\begin{equation} \label{eq:back_hj}
\left\{
\begin{array}{l}
-u_t + H(x,Du) = f ~~ \text{in} ~~ [0,T] \times \bb{T}^N,\\
u(T,x) = u_T(x)
\end{array} \right.
\end{equation}
where $\bb{T}^N$ is the $N$-dimensional torus $\bb{R}^N / \bb{Z}^N$. This setting is chosen mainly due to the ease of working with periodic boundary conditions. Time has been reversed solely to fit the customary notation for mean field games (see System (\ref{eq:mfg}) below). The Hamiltonian is defined by (\ref{eq:hamiltonian}). The unique viscosity solution of (\ref{eq:back_hj}) is given by the value function for the finite horizon optimal control problem
\begin{equation} \label{eq:value}
u(t,x) = \inf\{u_T(y_{x,t}^\alpha(T)) + \int_t^T f(s,y_{x,t}^\alpha(s))ds : \alpha \in L^\infty(t,T;A)\}
\end{equation}
where $y_{x,t}^\alpha$ is the solution of the open loop control problem
\begin{equation} \label{eq:open_loop1}
\left\{ \begin{array}{cc} \dot{y}(\tau) = c(y(\tau),\alpha(\tau)), & \tau > t \\ y(t) = x, & x \in \bb{T}^N \end{array} \right.
\end{equation}
We consider the following optimization problem:
\begin{problem} \label{pr:min_f}
Let the functional $\s{J}$ be defined on functions $f \in C([0,T] \times \bb{T}^N)$ by
\begin{equation}
\s{J}(f) = \int_0^T \int_{\bb{T}^N} K(t,x,f(t,x)) dx dt - \int_{\bb{T}^N} u(0,x) dm_0(x)
\end{equation}
where $u$ satisfies (\ref{eq:value}). Given the assumptions of Section \ref{sec:assumptions}, find $\inf_f \s{J}(f)$.
\end{problem}
Here the canonical example of the cost functional is $K(t,x,f(t,x)) = \frac{1}{p}|f(t,x)|^p$ (see Section \ref{sec:assumptions} below).

In general, we cannot expect the infimum in Problem \ref{pr:min_f} to be achieved, i.e. we cannot expect to find a continuous function $f$ such that $\s{J}(f)$ is the minimum. The goal of this section is to identify a function space in which we can expect to find minimizers of this problem, and indeed to find a relaxed setting in which the problem has a minimum. It may be noticed initially that a minimizing function $f$ should appear in the space $L^p$. However, this low-regularity does not allow for us to make sense of the value function (\ref{eq:value}). Additionally, we will have to make sense of (\ref{eq:back_hj}) in a sense even weaker than the sense of viscosity, since even continuity of solutions can be lost in general upon passage to the limit (see Section \ref{sec:loss_of_continuity}). Thus the issue of defining minimizers is rather delicate.

Following Cardaliaguet \cite{cardaliaguet13}, we might consider the following approach. Denote by $\s{K}_0$ the set of maps $u \in C^1([0,T] \times \bb{T}^N)$ such that $u(T,x) = u_T(x)$. Then define on $\s{K}_0$ the functional
\begin{equation}
\s{A}(u) = \int_0^T \int_{\bb{T}^N} K(t,x,-u_t + H(x,Du)) dx dt - \int_{\bb{T}^N} u(0,x) dm_0(x) .
\end{equation}
Then we have
\begin{problem}[Smooth problem] \label{pr:smooth}
Find $\inf_{u \in \s{K}_0} \s{A}(u)$.
\end{problem}
This problem can be relaxed as follows. Let $\s{K}$ be the set of pairs $(u,f) \in BV((0,T) \times \bb{T}^N) \times L^p((0,T) \times \bb{T}^N)$
such that $u(T,x) = u_T(x)$ in the sense of traces, and
\begin{equation}
-u_t + H(x,Du) \leq f ~~~ \text{in} ~(0,T) \times \bb{T}^N
\end{equation}
in the sense of distributions. Note that $\s{K}$ is a convex set which has been chosen based on the regularity we are able to obtain by passing to the limit on sequences of minimizers (see Theorem \ref{thm:existence}). We have that $\s{K}_0$ embeds naturally $\s{K}$, since for $u \in \s{K}_0$ the pair $(u,-u_t + H(x,Du)) \in \s{K}$. Define on $\s{K}$ the functional
\begin{equation}
\s{A}(u,f) = \int_0^T \int_{\bb{T}^N} K(t,x,f(t,x)) dx dt - \int_{\bb{T}^N} u(0,x) dm_0(x).
\end{equation}
\begin{problem}[Relaxed problem] \label{pr:relaxed}
Find $\inf_{(u,f) \in \s{K}} \s{A}(u,f).$
\end{problem}
The next step would be to show directly that the smooth and relaxed problems have the same infimum. However, this requires assumption on the Hamiltonian structure which explicitly excludes cases such as $c(x)|Du|$ (see \cite{cardaliaguet13}). The reason for this is that the method of proof is built on using classical convolution methods to recover smooth solutions to the Hamilton-Jacobi equations, and this can only work if there are no variable coefficients on $|Du|$ which get in the way.

To illustrate the difficulty, let us give an outline of a potential proof that the smooth problem \ref{pr:smooth} has the same infimum as the relaxed problem given above. Since the purpose is simply illustration, we take for now $H(x,Du) = c(x)|Du|$. It is enough to show that $\inf_{(u,f) \in \s{K}} \s{A}(u,f) \geq \inf_{u \in \s{K}_0} \s{A}(u)$. With this in mind, let $(u,f) \in \s{K}$.
Let $\xi \geq 0$ be a smooth convolution kernel in $\bb{R}^{N+1}$ with support in the unit ball and $\int \xi = 1$. Set $\xi_\epsilon(t,x) = \epsilon^{-N-1}\xi((t,x)/\epsilon)$, $u_\epsilon = \xi_\epsilon \ast u$ and $f_\epsilon = \xi_\epsilon \ast f$. Then
\begin{equation}
-\partial_t u_\epsilon + \xi_\epsilon \ast H(\cdot,Du) = f_\epsilon.
\end{equation}
We want to deduce from this equation an estimate of the form $-\partial_t u_\epsilon + H(x,Du_\epsilon) \leq \tilde{f}_\epsilon$, where $\tilde{f}_\epsilon$ is an appropriate modification of $f_\epsilon$ whose $L^p$ norm also converges to the $L^p$ norm of $f$. By convexity of $H$ in the second variable we have
\begin{equation}
H(x,Du_\epsilon(t,x)) \leq \xi_\epsilon \ast H(\cdot,Du) + \beta_\epsilon(t,x)
\end{equation}
where
\begin{equation}
\beta_\epsilon(t,x) = \int_{B_\epsilon(t,x)} \xi_\epsilon(t-s,x-y)|H(y,Du(s,y)) - H(x,Du(s,y))| ds dy.
\end{equation}
In order to estimate this term, one could use the fact that $c$ is Lipschitz to get
\begin{equation}
\beta_\epsilon(t,x) \leq  L \epsilon \int_{B_\epsilon(t,x)} \xi_\epsilon(t-s,x-y)|Du(s,y)|ds dy.
\end{equation}
But this is too crude a bound, because from here there is little we can do to estimate the remaining term. Indeed, with only $L^1$ regularity on $Du$, and at best a bound on $\xi_\epsilon$ of the form $C\epsilon^{-d-1}$, this means our estimate is on the order of $\epsilon^{-d}$, which diverges. Thus the smooth version of $u$ is no longer guaranteed to be a subsolution to the Hamilton-Jacobi equation.

The failure of the above argument in this case is in some sense a geometric problem. The trajectories over which the optimal control problem \ref{eq:value} is defined are constrained with respect to a metric with variable coefficients in space. This suggested to us a Riemannian geometry approach. The articles of Greene and Wu \cite{greene73,greene74} from the 1970's feature a method of convolution on Riemannian manifolds which preserves geometric properties of functions, such as Lipschitz continuity and convexity. This inspired some hope that one could use a modified convolution argument to prove the desired inequality above. However, we were unable to make this work.

Instead, our argument will avoid entirely the use of convolution smoothing of weak (sub-)solutions to the Hamilton-Jacobi equations. 
We exploit the fact (shown below in Theorem \ref{thm:dual}) that the optimal control of the Hamilton-Jacobi equations is in duality with that of the continuity equation.
To obtain a relaxed problem which has the same infimum as the smooth problem \ref{pr:smooth}, it suffices to show that both smooth and weak Hamilton-Jacobi subsolutions have the same basic relationship to the {\em dual} variable, which is a measure characterized as a generalized ODE flow (by the superposition principle, see Section \ref{sec:original_smooth}).
This is natural if one considers that $u$ is supposed to be the value function of an ODE optimal control problem.
Additionally, we gain the advantage the continuity equation has more tractable structure than the Hamilton-Jacobi equation, and it is relatively simple to obtain approximations of solutions by smooth functions satisfying the same equation.

The structure of this section is thus as follows. After stating our assumptions, we will analyze the smooth problem and its dual. The importance of the dual problem is twofold: in the first place it gives the adjoint equation for the characterization of minimizers (\ref{eq:mfg}), and in the second it provides the target minimum value which we expect to attain. We then show that the original problem \ref{pr:min_f} has the same infimum as the smooth problem. This is followed by a discussion of the loss of continuity, which is in contrast to \cite{cardaliaguet13} and is a direct consequence of the linear (as opposed to superlinear) growth of the Hamiltonian. Finally, we give the correct definition of the relaxed problem, prove the existence of a minimum, and show that this minimum is precisely the infimum of the original problem.

\subsection{Assumptions} 
\label{sec:assumptions}

\begin{enumerate}
\item $dm_0(x)$ is a positive finite measure on $\bb{T}^N$ which is absolutely continuous with respect to Lebesgue measure, having density (which we also call $m_0$) in $L^\infty(\bb{T}^N)$. (Most examples which inspire our study satisfy this assumption, for example, we could let $m_0(x)$ be the indicator of some set $\Omega$, in which case minimizing $- \int u(0,x)m_0(x)dx$ maximizes the probability that the fire is contained inside the set $\Omega$.)
\item We assume that set of controls $A$ is a compact topological space, and $c : \bb{T}^N \times A \to \bb{R}^N$ is a continuous vector field which is Lipschitz in the first variable. We assume that $c(x,A)$ is convex for all $x \in \bb{T}^N$. Moreover there exists a constant $c_0 > 0$ such that $c(x,A) \supset B_{c_0}(0)$ for all $x \in \bb{T}^N$.

The above requirements on $c$ imply the following conditions on $H$. First, we have that $H$ is Lipschitz in the first variable:
\begin{equation} \label{eq:hamiltonian_lipschitz}
|H(x,p) - H(y,p)| \leq L|x-y||p|.
\end{equation}
Second, we have the following linear bounds:
\begin{equation}
\label{eq:hamiltonian_bounds}
c_0|p| \leq H(x,p) \leq c_1|p|,
\end{equation}
where $c_1 := \sup_{x,a} |c(x,a)|$. Note also that $H(x,\cdot)$ is subadditive and positively homogeneous, hence convex. We denote by $H^*(x,\cdot)$ the Fenchel conjugate of $H(x,\cdot)$, which, due to the above assumptions, is simply the indicator function on the convex set $-c(x,A)$, i.e.
\begin{equation} \label{eq:hamiltonian_conjugate}
H^*(x,q) = \left\{\begin{array}{ccc}
0 & \text{if} & q \in -c(x,A),\\
\infty & \text{if} & q \notin -c(x,A).
\end{array} \right.
\end{equation}
\item $u_T$ is a Lipschitz continuous function on $\bb{T}^N$.
\item The cost function $K = K(t,x,f)$ is continuous in all variables, continuously differentiable and strictly convex in $f$, and satisfies the growth conditions
\begin{equation} \label{eq:cost_growth}
\frac{1}{C}|f|^p - C \leq K(t,x,f) \leq C|f|^p + C
\end{equation}
for a fixed $p > N+1$. We will assume moreover that $K(t,x,0) = 0$ and that $K(t,x,f)$ is increasing in $f$ for $f \geq 0$. We denote by $K^*(t,x,\cdot)$ the Fenchel conjugate of $K(t,x,\cdot)$, and we will denote by $k$ the partial derivative of $K^*$ with respect to the dual variable, i.e. $k(t,x,m) = \partial_m K^*(t,x,m)$. Note that $k$ is continuous in all variables and increasing in $m$. 
\end{enumerate}
We will denote throughout the conjugate exponent of $p$ by $q = p^*$. Observe that (\ref{eq:cost_growth}) is equivalent to the following conditions on $K^*$ and $k$:
\begin{equation} \label{eq:cost_growth2}
\frac{1}{C}|m|^q - C \leq K^*(t,x,m) \leq C|m|^q + C
\end{equation}
and
\begin{equation} \label{eq:cost_growth3}
\frac{1}{C}|m|^{q-1} - C \leq k(t,x,m) \leq C|m|^{q-1} + C
\end{equation}
for appropriate changes in the (universal) constant $C$.

\subsection{Smooth problem and its dual}
\label{sec:smooth_dual}

One immediate advantage of Problem \ref{pr:smooth} is that one can analyze it by quickly computing the dual problem. In this subsection we identify the dual problem \ref{pr:dual} as a minimization over solutions to a continuity equation, and we show that the problem has a minimum. In the rest of the section we use this problem essentially as a target. That is, when we define the relaxed problem, we want to show that it is also in duality with Problem \ref{pr:dual}.

Define $\s{K}_1$ to be the set of all pairs $(m,{\bf w}) \in L^1((0,T) \times \bb{T}^N) \times L^1((0,T) \times \bb{T}^N;\bb{R}^N)$ such that $m \geq 0$ almost everywhere, ${\bf w}(t,x) \in m(t,x)c(x,A)$ almost everywhere, and
\begin{align*}
\partial_t m + \mathrm{div}~({\bf w}) &= 0\\
m(0,\cdot) &= m_0(\cdot)
\end{align*}
in the sense of distributions. Because of the integrability assumption on $\bf w$, it follows that $t \mapsto m(t)$ has a unique representative such that $\int m(t)\phi$ is continuous on $[0,T]$ for all $\phi \in C^\infty(\bb{T}^N)$ (cf. \cite{ambrosio08}). It is to this representative that we refer when we write $m(t)$, and thus $m(t)$ is well-defined for all $t \in [0,T]$.

Note that $\s{K}_1$ is convex by the assumption that $c(x,A)$ is convex for all $x$. We define the {\em dual problem} as follows. Define a functional
\begin{equation}\label{eq:dual}
\s{B}(m,{\bf w}) = \int_{\bb{T}^N} u_T(x)m(T,x)dx + \int_0^T \int_{\bb{T}^N} K^*(t,x,m(t,x)) dx dt
\end{equation}
on $\s{K}_1$. 
Since $m \geq 0$, we have that the second integral in (\ref{eq:dual}) is well-defined in $(-\infty,\infty]$ by the assumptions on $K$. The first integral is well-defined by the continuity of $u_T$ and the fact that $m(T,x)dx$ is a finite measure.

We next state the ``dual problem" succinctly as
\begin{problem}[Dual Problem] \label{pr:dual}
Find $\inf_{(m,{\bf w}) \in \s{K}_1} \s{B}(m,{\bf w})$.
\end{problem}
Note that we could equally well define $\s{B}$ as
\begin{equation}\label{eq:dual2}
\s{B}(m,{\bf w}) = \int_{\bb{T}^N} u_T(x)m(T,x)dx + \int_0^T \int_{\bb{T}^N} K^*(t,x,m(t,x)) + \chi(t,x,{\bf w}(t,x)) dx dt
\end{equation}
where
\begin{equation}
\chi(t,x,{\bf w}) = \left\{\begin{array}{ccc} 0 & \text{if} & {\bf w} \in m(t,x)c(x,A), \\ \infty &\text{if} & {\bf w} \notin m(t,x)c(x,A). \end{array} \right.
\end{equation}
In this case we could lift the restriction ${\bf w}(t,x) \in m(t,x)c(x,A)$ in the definition of $\s{K}_1$. The problem of minimizing $\s{B}$ remains unchanged. We therefore treat these two formulations interchangeably, although it is the second which makes more explicit how the functional $\s{B}$ depends on $\bf w$.

The main result of this subsection is
\begin{theorem}
\label{thm:dual} Problems \ref{pr:smooth} and \ref{pr:dual} are in duality, i.e.
\begin{equation}
\inf_{u \in \s{K}_0} \s{A}(u) = -\min_{(m,{\bf w}) \in \s{K}_1} \s{B}(m,{\bf w})
\end{equation}
Moreover, the minimum on the right hand side is achieved by a pair $(m,{\bf w}) \in \s{K}_1$, of which $m$ is unique, and which must satisfy the following: $(t,x) \mapsto K^*(t,x,m(t,x)) \in L^1((0,T) \times \bb{T}^N)$, $m \in L^q((0,T) \times \bb{T}^N)$, and ${\bf w}(t,x) \in m(t,x)c(x,A)$ almost everywhere.
\end{theorem} 
\begin{remark}
In general, we cannot expect $\bf w$ to be unique because the functional $\s{B}$ is not strictly convex in $\bf w$.
\end{remark}

The proof of Theorem \ref{thm:dual} is essentially an application of the Fenchel-Rockafellar theorem, which is stated as follows.
\begin{theorem}[Fenchel-Rockafellar]
Let $X$ and $Y$ be Banach spaces and denote by $X'$ and $Y'$ their duals. Suppose $\Lambda : X \to Y$ is a continuous linear operator with adjoint $\Lambda^* : Y' \to X'$. Finally, let $\sr{F}$ and $\sr{G}$ be convex functionals on $X$ and $Y$, respectively, with respective Fenchel conjugates denoted by $\sr{F}^*$ and $\sr{G}^*$. Then
\begin{equation}
\inf_{x \in X}\sr{F}(x) + \sr{G}(\Lambda x) = \sup_{y' \in Y'} -\sr{F}^*(\Lambda^*(y')) - \sr{G}^*(-y').
\end{equation}
\end{theorem}
In order to prove Theorem \ref{thm:dual}, we need to define function spaces used in the application of the Fenchel-Rockefeller duality Theorem. 
Let
\begin{itemize}
\item $X := C^1([0,T] \times \bb{T}^N;\bb{R})$,
\item $Y := C([0,T] \times \bb{T}^N;\bb{R}) \times C([0,T] \times \bb{T}^N;\bb{R}^N)$.
\end{itemize}
Then define the functionals $\sr{F}:X \to \bb{R}$ and $\sr{G}:Y \to \bb{R}$ by
\begin{equation}
\begin{array}{l}
\sr{F}(u) = \left\{\begin{array}{ccc} -\int_{\bb{T}^N} u(0,x)dm_0(x) & \text{if} & u(T,\cdot) = u_T(\cdot)\\ \infty & \text{otherwise} & \end{array}\right.\\
\sr{G}(a,{\bf b}) = \int_0^T \int_{\bb{T}^N} K(t,x,-a+H(x,{\bf b})) dx dt
\end{array}
\end{equation}
and let $\Lambda : X \to Y$ be the linear map given by $\Lambda(u) = (\partial_t u, Du)$. Let $\sr{F}^*$ and $\sr{G}^*$ represent the conjugate of $\sr{F}$ and $\sr{G}$ respectively.
\begin{proposition} \label{prop:conjugates}
For all $(m,{\bf w}) \in Y'$,
\begin{align}
\sr{F}^*(\Lambda^*(m,{\bf w})) &= \left\{ \begin{array}{ccc} \int_{\bb{T}^N}u_T(x)m(T,x)dx & \text{if} & \begin{array}{c} \partial_t m + \mathrm{div}~ {\bf w} = 0\\ m(0,\cdot) = m_0(\cdot) \end{array} \\ \infty & \text{otherwise} & \end{array} \right.\\
\sr{G}^*(-m,-{\bf w}) &=\left\{\begin{array}{ccc} \int_0^T \int_{\bb{T}^N} K^*(t,x,m(t,x)) dx dt & \text{if} & \begin{array}{c} m \in L^q([0,T] \times \bb{T}^N), m \geq 0 ~\mathrm{a.e.} \\ {\bf w}(t,x) \in m(t,x)c(x,A) ~\mathrm{a.e.} \end{array}\\ \infty & \text{otherwise} & \end{array} \right.
\end{align}
where the equation $\partial_t m + \mathrm{div}~{\bf w} = 0$ holds in the sense of distributions.
\end{proposition}

\begin{proof}
At first we have
\begin{align*}
\sr{G}^*(-m,-{\bf w}) &= \sup_{(a,{\bf b}) \in Y} \langle -(a,{\bf b}),(m,{\bf w}) \rangle - \sr{G}(a,{\bf b}) \\
&= \sup_{(a,{\bf b}) \in Y} \int_{(0,T) \times \bb{T}^N} -am - {\bf b} \cdot {\bf w} - K(t,x,-a+H(x,{\bf b})) dx dt\\
&= \sup_{(a,{\bf b}) \in Y} \int_{(0,T) \times \bb{T}^N} -{\bf b} \cdot {\bf w} - H(x,{\bf b})m + m(-a + H(x,{\bf b})) - K(t,x,-a+H(x,{\bf b})) dx dt\\
&= \sup_{(a,{\bf b}) \in Y} \int_{(0,T) \times \bb{T}^N} -{\bf b} \cdot {\bf w} - H(x,{\bf b})m + ma - K(t,x,a) dx dt\\
&= \sup_{{\bf b}} \int_{(0,T) \times \bb{T}^N} -{\bf b} \cdot {\bf w} - H(x,{\bf b})m + K^*(t,x,m(t,x)) dx dt
\end{align*}
Suppose first that $m \geq 0$ and ${\bf w}(t,x) \in m(t,x)c(x,A)$ almost everywhere. Let
\begin{equation}
{\bf v}(t,x) = \left\{ \begin{array}{ccc} \displaystyle \frac{{\bf w}(t,x)}{m(t,x)} & \text{if} & m(t,x) > 0, \\ 0 &\text{if}& m(t,x) = 0. \end{array} \right.
\end{equation}
We have that $m{\bf v} = {\bf w}$ and, since $c(x,A)$ always contains the zero vector, that ${\bf v}(t,x) \in c(x,A)$ almost everywhere. Thus $H^*(x,-{\bf v}) = 0$ and in particular $-{\bf b} \cdot {\bf w} - H(x,{\bf b})m \leq 0$ for any ${\bf b}$. It follows that
\begin{equation} \label{eq:g_star}
\sr{G}^*(-m,-{\bf w}) \leq \int_{(0,T) \times \bb{T}^N} K^*(t,x,m(t,x))dt dx.
\end{equation}
On the other hand, this upper bound can be achieved by taking ${\bf b} = 0$. This proves that (\ref{eq:g_star}) is an equality when ${\bf w}(t,x) \in m(t,x)c(x,A)$ almost everywhere.

We next turn to the case where $m \leq -\epsilon < 0$ on some set $V$ of nonzero measure (otherwise, we have $m \geq 0$ almost everywhere). Then we have
\begin{align*}
\sr{G}^*(-m,-{\bf w}) &\geq \sup_{{\bf b}} \int_{V} -{\bf b} \cdot {\bf w} - H(x,{\bf b})m + K^*(t,x,m(t,x)) dx dt\\
&\geq \sup_{{\bf b}} \int_{V} -{\bf b} \cdot {\bf w} + \epsilon c_0|{\bf b}| dx dt = \infty.
\end{align*}

Finally, we consider $m \geq 0$ almost everywhere but ${\bf w}(t,x) \notin m(t,x)c(x,A)$ on some set $U$ of nonzero measure. Note that ${\bf w} \neq 0$ on $U$ since $c(x,A)$ always contains the zero vector. Observe that either $U_1 := U \cap \{m = 0\}$ or $U_2 := U \cap \{m > 0\}$ has nonzero measure. In the former case, take a supremum over continuous vector fields $\bf b$ with support in $U_1$ to get
\begin{equation}
\sup_{{\bf b}} \int_{(0,T) \times \bb{T}^N} -{\bf b} \cdot {\bf w} - H(x,{\bf b})m ~dx dt \geq \sup_{{\bf b}} \int_{U_1} -{\bf b} \cdot {\bf w} dx dt = \infty.
\end{equation}
In the latter case, we use the fact that $H^*(x,-\frac{{\bf w}}{m}) = \infty$ to see that
\begin{equation}
\sup_{{\bf b}} \int_{(0,T) \times \bb{T}^N} -{\bf b} \cdot {\bf w} - H(x,{\bf b})m ~dx dt \geq \sup_{{\bf b}} \int_{U_2} m\left({\bf b} \cdot \left(-\frac{{\bf w}}{m}\right) - H(x,{\bf b})\right) dx dt = \infty.
\end{equation}
This completes the computation of $\sr{G}^*$.

For $\sr{F}^*$, we use integration by parts to deduce that
\begin{align*}
\sr{F}^*(\Lambda^*(m,{\bf w})) &= \sup_{u \in X} \langle \Lambda u, (m,{\bf w}) \rangle - \sr{F}(u) \\
&= \sup_{u \in X, u(0) = g} \left( \int_{(0,T) \times \bb{T}^N} m\partial_t u + Du \cdot {\bf w} dx dt + \int_{\bb{T}^N} u(0,x)dm_0(x) \right)\\
&= \sup_{u \in X, u(0) = g} \left( -\int_{(0,T) \times \bb{T}^N} u(\partial_t m + \mathrm{div}~ {\bf w}) dx dt \right. \\
&~~~~~~ \left. + \int_{\bb{T}^N} u_T(x)m(T,x) - u(0,x)m(0,x) dx + \int_{\bb{T}^N} u(0,x)d
m_0(x) \right).
\end{align*}
Now it follows by a usual series of arguments that unless
\begin{equation*}
\begin{array}{c}
\partial_t m + \mathrm{div}~ {\bf w} = 0\\
m(0,\cdot) = m_0
\end{array}
\end{equation*}
the supremum above is infinite. This can be seen by choosing $u$ appropriately: if $\partial_t m + \nabla \cdot {\bf w} \neq 0$ then $-\int u(\partial_t m + \nabla \cdot {\bf w})$ can be made aribtrarily large, and we can simply choose $u(0,\cdot) = 0$ in this case. On the other hand, if $\partial_t m + \nabla \cdot {\bf w} = 0$ but $m(T,\cdot) \neq m_0$, then $\int u(0,x)(m(0,x)dx-dm_0(x))$ can be made arbitrarily large. This completes the proof.

\end{proof}

With the above proposition, the proof of Theorem \ref{thm:dual} is straightforward.
\begin{proof}[Proof of Theorem \ref{thm:dual}]
Problem \ref{pr:smooth} can be written as
\begin{equation}
\inf_{u \in X} \sr{F}(u) + \sr{G}(\Lambda(u))
\end{equation}
which by the Fenchel-Rockefeller theorem is equal to
\begin{equation}
\max_{(m,{\bf w}) \in Y'} -\sr{F}^*(\Lambda^*(m,{\bf w})) - \sr{G}^*(-(m,{\bf w})) = -\min_{(m,{\bf w}) \in Y'} \sr{F}^*(\Lambda^*(m,{\bf w})) + \sr{G}^*(-(m,{\bf w})).
\end{equation}
In particular, we have an optimizing pair $(m,{\bf w}) \in Y'$. By Proposition \ref{prop:conjugates} we deduce that $m \in L^q((0,T) \times \bb{T}^N)$, ${\bf w}(t,x) \in m(t,x)c(x,A)$ almost everywhere, and
\begin{equation}
m_t + \mathrm{div}~{\bf w} = 0, ~~ m(0,\cdot) = m_0.
\end{equation}
In particular, $(m,{\bf w}) \in \s{K}_1$. Moreover, taking into account that $(m,{\bf w})$ minimizes the functional $\s{B}$, we deduce that $K^*(\cdot,\cdot,m) \in L^1((0,T) \times \bb{T}^N)$. Uniqueness of $m$ follows from the fact that $\s{K}_1$ is a convex set and that $\s{B}(m,{\bf w})$ is strictly convex in $m$.
\end{proof}

\subsection{Original problem is equivalent to smooth problem}
\label{sec:original_smooth}

In this subsection we prove
\begin{proposition} \label{prop:original}
The original problem \ref{pr:min_f} is equivalent to the smooth problem \ref{pr:smooth}, that is, the two problems have the same infimum. Equivalently, Problem \ref{pr:min_f} is in duality with Problem \ref{pr:dual}.
\end{proposition}

To prove Proposition \ref{prop:original}, it suffices to show that $\inf_{f \in C([0,T] \times \bb{T}^N)} \s{J}(f) \geq - \inf_{(m,{\bf w}) \in \s{K}_1} \s{B}(m,{\bf w})$. Indeed, the opposite inequality has already been proved by Theorem \ref{thm:dual}.

Let us first outline how to proceed. For a vector field $\bf v$, let $X_{\bf v}(t,x)$ be the flow corresponding to the ODE
\begin{equation} \label{eq:ode}
\dot{X}(t) = {\bf v}(t,X(t)), ~~ X(0) = x.
\end{equation}
Provided $\bf v$ is regular enough (say, Lipschitz in the space variable) we have from standard Cauchy-Lipschitz theory that the pushforward $m(t,\cdot) = X(t,\cdot) \# \s{L}^N$, where $\s{L}^N$ is Lebesgue measure, is the solution to the continuity equation
\begin{equation} \label{eq:continuity}
m_t + \mathrm{div}(m{\bf v}) = 0, ~~ m(0)= m_0.
\end{equation}
Suppose, moreover, that ${\bf v}(t,x) \in c(x,A)$ almost everywhere. Then the trajectory $X_{\bf v}(\cdot,x)$ is in fact a solution of the open loop controlled ODE (\ref{eq:open_loop1}). With this in mind and taking into account formula (\ref{eq:value}), we perform the following formal computations:
\begin{align*}
\inf_f \s{J}(f) 
& = \inf_f \int_0^T \int_{\bb{T}^N} K(t,x,f(t,x)) dx dt - \int_{\bb{T}^N} \inf_y \{ u_T(y(T)) + \int_0^T f(t,y(t))dt \}m_0(x) dx\\
& \geq \inf_f \int_0^T \int_{\bb{T}^N} K(t,x,f(t,x)) dx dt - \inf_{\bf v}\int_{\bb{T}^N}  u_T(X_{\bf v}(T,x)) + \int_0^T f(t,X_{\bf v}(t,x))dt ~m_0(x) dx\\
& = \inf_f\sup_{m,{\bf v}} \int_0^T \int_{\bb{T}^N} K(t,x,f(t,x)) dx dt - \int_{\bb{T}^N}  u_T(x)m(T,x)dx - \int_{\bb{T}^N}\int_0^T f(t,x)m(t,x)dt dx\\
& = \sup_{m,{\bf v}} -\int_0^T \int_{\bb{T}^N} K^*(t,x,m(t,x)) dx dt - \int_{\bb{T}^N}  u_T(x)m(T,x)dx = -\inf_{m,{\bf v}} \s{B}(m,m{\bf v}).
\end{align*}
Of course, the above calculation is merely formal, since exchanging the infimum over trajectories for an infimum over (regular) vector fields is hardly straightforward. Nevertheless, the calculation is suggestive, in that it suggest a direct link between the original and dual problems (no smoothness required on $u$).

In order to justify the above calculations when $\bf v$ is not regular, we need a result from transport theory that allows us to define the flow at least in a {\em probabilistic} sense. This result is called the {\em superposition principle}. We first define a {\em superposition solution,} and then give the statement of the corresponding principle.
\begin{definition}[Superposition solution] \label{def:superposition}
Let $m(t)$ be a positive measure-valued solution of the continuity equation (\ref{eq:continuity}). We say that $m$ is a {\em superposition solution} if 
there exists a positive measure $\eta$ on the set $C([0,T] \times \bb{T}^N)$ such that $\eta$ is concentrated on integral curves of the ODE \ref{eq:ode}, i.e.
\begin{equation}
\int_{C([0,T] \times \bb{T}^N)} \left|\gamma(t) - x - \int_0^t{\bf v}(s,\gamma(s))ds \right| d\eta(\gamma) = 0,
\end{equation}
and such that $m(t)$ is the pushforward of $\eta$ under the evaluation function $\gamma \mapsto \gamma(t)$, i.e.
\begin{equation}
\int_{\bb{T}^N} \phi(x)dm(t,x) = \int_{C([0,T] \times \bb{T}^N)} \phi(\gamma(t)) d\eta(\gamma).
\end{equation}
\end{definition}
\begin{theorem}[Superposition principle] \label{thm:superposition}
Let $m(t)$ be a positive measure-valued solution of the continuity equation (\ref{eq:continuity}), and assume that
\begin{equation}
\int_0^T \int_{\bb{T}^N} \frac{|{\bf v}(t,x)|}{1+|x|} dm(t,x)dt < \infty.
\end{equation}
Then $m$ is a {\em superposition solution}.
\end{theorem}
For a proof of this result, see \cite{ambrosio08}.

We now prove the following crucial lemma.
\begin{lemma} \label{lem:ibp_inequality}
Suppose ${\bf v}$ is a vector field satisfying ${\bf v}(t,x) \in c(x,A)$ and $m \in L^1$ satisfies the continuity equation (\ref{eq:continuity}) in the sense of distributions. Let $f$ be a continuous function, and let $u$ be the corresponding value function given by Equation (\ref{eq:value}). Then for $0 \leq t \leq s \leq T$ we have
\begin{equation} \label{eq:integration_by_parts}
\int_{\bb{T}^N} u(s,x)m(s,x) - u(t,x)m(t,x)dx  + \int_t^s \int_{\bb{T}^N}  f(\tau,x)m(\tau,x)dx d\tau \geq 0.
\end{equation}
\end{lemma}

\begin{proof}
Because $c(x,A)$ is uniformly bounded, it follows that $\bf v \in L^\infty$. Since $m \in L^1$, we have that $\bf v$ satisfies the integrability condition of the superposition principle. We apply Theorem \ref{thm:superposition} to see that $m(t)$ is the push-forward under the evaluation map $t \mapsto \gamma(t)$ of some positive measure  $\eta$ on $C([0,T];\bb{T}^N)$ which is concentrated on trajectories of ${\bf v}$, i.e. solutions of the ODE (\ref{eq:ode}). This means that for any continuous function $\phi$,
\begin{equation}
\int_{\bb{T}^N} \phi(x)m(t,x)dx = \int_{C([0,T];\bb{T}^N)} \phi(\gamma(t))d\eta(\gamma)
\end{equation}

We recall that $u$ satisfies the {\em dynamic programming principle} \cite{bardi97}:
\begin{equation} \label{eq:dpp}
u(t,x) = \inf\{u(s,y_{x,t}^\alpha(s)) + \int_t^s f(\tau,y_{x,t}^\alpha(\tau))ds : \alpha \in L^\infty(t,s;A)\},
\end{equation}
where we recall that $y_{x,t}^\alpha$ is the solution to the ODE (\ref{eq:open_loop1}). Appealing directly to formula (\ref{eq:dpp}) and the fact that ${\bf v}(t,x) \in c(x,A)$ a.e., it follows that for each trajectory $\gamma$ of the vector field $\bf v$ we have
\begin{equation}
u(t,\gamma(t)) \leq u(s,\gamma(s)) + \int_t^s f(\tau,\gamma(\tau))d\tau.
\end{equation}
Integrating over $C([0,T];\bb{T}^N)$ and using the fact that $\eta$ is concentrated on such trajectories, we see (after using the Fubini theorem to switch the order of integration) that
\begin{align}
\int_{\bb{T}^N} u(t,x)m(t,x)dx &= \int_{C([0,T];\bb{T}^N)} u(t,\gamma(t))d\eta(\gamma) \\
&\leq \int_{C([0,T];\bb{T}^N)} \left(u(s,\gamma(s)) + \int_t^s f(\tau,\gamma(\tau))d\tau\right) d\eta(\gamma)\\
&= \int_{\bb{T}^N} u(s,x)m(s,x)dx + \int_t^s \int_{\bb{T}^N}f(\tau,x)m(\tau,x)dx d\tau,
\end{align}
as desired.
\end{proof}

Lemma \ref{lem:ibp_inequality} allows us to prove Proposition \ref{prop:original}. However, its significance also lies in the information it gives us about solutions $u$ which, as we will show, can be retained when passing to the limit on a sequence of minimizers. Thus the lemma plays a crucial role in the construction of the relaxed problem below.

\begin{proof}[Proof of Proposition \ref{prop:original}]
We recall that Problem \ref{pr:min_f} is to minimize
\begin{equation}
\s{J}(f) = \int_0^T \int_{\bb{T}^N} K(t,x,f(t,x)) dx dt - \int_{\bb{T}^N} u(0,x) dm_0(x)
\end{equation}
where $u$ is given by (\ref{eq:value}).
On the one hand, we have that
\begin{equation} \label{eq:less_than}
\inf_f \s{J}(f) \leq \inf_{u \in \s{K}_0}\s{A}(u).
\end{equation}
Indeed, if $u \in \s{K}_0$, then setting
\begin{equation}
f = -u_t + H(x,Du),
\end{equation}
it follows that $f$ is a continuous function and that therefore $u$ is the unique viscosity solution to the PDE. Hence $u$ also satisfies (\ref{eq:value}) and thus $\s{J}(f) \leq \s{A}(u)$. 

It remains to prove the other inequality, which by Theorem \ref{thm:dual} is
\begin{equation} \label{eq:greater_than}
\inf_f \s{J}(f) \geq \inf_{u \in \s{K}_0}\s{A}(u) = -\min_{(m,{\bf w}) \in \s{K}_1} \s{B}(m,{\bf w}).
\end{equation}
In particular, if $(m,{\bf w}) \in \s{K}_1$ is a minimizer of Problem \ref{pr:dual}, then it is enough to show that
\begin{equation}
\s{J}(f) \geq -\s{B}(m,{\bf w}) ~~\text{for all} ~ f \in C([0,T] \times \bb{T}^N).
\end{equation}
By Lemma \ref{lem:ibp_inequality} we have
\begin{align}
\s{J}(f) &= \int_0^T \int_{\bb{T}^N} K(t,x,f(t,x)) dx dt - \int_{\bb{T}^N} u(0,x)m_0(x)dx\\
&\geq \int_0^T \int_{\bb{T}^N} K(t,x,f(t,x)) dx dt - \int_{\bb{T}^N} u_T(x)m(T,x)dx - \int_0^T \int_{\bb{T}^N}f(t,x)m(t,x)dx dt\\
&\geq -\int_0^T \int_{\bb{T}^N} K^*(t,x,m(t,x)) dx dt - \int_{\bb{T}^N} u_T(x)m(T,x)dx = -\s{B}(m,{\bf w}).
\end{align}
This completes the proof of Lemma \ref{prop:original}.

\end{proof}

\subsection{Relaxed problem and its minimizers}
\label{sec:relaxed}

As previously stated, the goal of this section is to construct a set $\tilde{\s{K}}$ on which a relaxed problem can be appropriately defined so that there exists a minimizer $(u,f) \in \tilde{\s{K}}$ of the functional $\s{A}(u,f)$. We will define that space here, but first we take care to define some necessary preliminaries.

Given $u \in BV$, there is a Radon measure $\mu$ and a $\mu$-measurable vector field $\sigma = \sigma(t,x)$ such that $|\sigma(t,x)| = 1$ for $\mu$-almost everywhere $(t,x)$ and
\begin{equation}
\int_0^T \int_{\bb{T}^N} u(t,x)(\phi_t(t,x) + \mathrm{div} ~ \phi(t,x))dt dx = - \int_0^T \int_{\bb{T}^N} \phi(t,x) \cdot \sigma(t,x) d\mu(t,x)
\end{equation}
for smooth vector fields $\phi$ with compact support in $(0,T) \times \bb{T}^N$.
(See, for example, \cite{evans92}). The vector measure $\sigma \mu$ is the distributional derivative of $u$. We can write $\sigma(t,x) = (\sigma_1(t,x),\sigma_N(t,x))$, $\sigma_1$ a real-valued function and $\sigma_N$ a $\bb{R}^N$-valued function (and both $\mu$-measurable). We define $\partial_t u := \sigma_1 d\mu$ and $Du := \sigma_N d\mu$, i.e.
\begin{align*}
\int_0^T \int_{\bb{T}^N} \phi(t,x)u_t(t,x) &:= \int_0^T \int_{\bb{T}^N} \phi(t,x)\sigma_1(t,x)d\mu(t,x),\\
\int_0^T \int_{\bb{T}^N} {\bf \psi}(t,x)\cdot Du(t,x) &:= \int_0^T \int_{\bb{T}^N} \psi(t,x)\cdot \sigma_N(t,x)d\mu(t,x)
\end{align*}
for continuous functions $\phi$ and vector fields $\psi$. In particular, if $\phi,\psi$ are smooth with compact support in $(0,T) \times \bb{T}^N$, then it follows from the above definitions that
\begin{equation}
\int_0^T \int_{\bb{T}^N} u(t,x)(\phi_t(t,x) + \mathrm{div}~\psi(t,x))dt dx = - \int_0^T \int_{\bb{T}^N} \phi(t,x)u_t(t,x) + \psi(t,x) \cdot Du(t,x).
\end{equation}
We define $|Du| = |\sigma_N| d\mu$. Note that $Du$ and $|Du|$ are Radon measures on $[0,T] \times \bb{T}^N$ (the first being a vector-measure). Let $\frac{Du}{|Du|}$ denote the Radon-Nikodym derivative of $Du$ with respect to $|Du|$. We have, in fact,
\begin{equation}
\frac{Du}{|Du|} = \left\{\begin{array}{ccc}
\frac{\sigma_N}{|\sigma_N|} & \text{on} & \{\sigma_N \neq 0\}\\
0 & \text{on} & \{\sigma_N = 0\}
\end{array}
\right.
\end{equation}
The set $\{\sigma_N = 0\}$ is of course $|Du|$-negligible, and $\frac{Du}{|Du|}$ is equal to one $|Du|$-almost everywhere.

We will denote by $H(x,Du)$ the measure $H(x,\sigma_N)d\mu$. Since $H$ is positively homogeneous in the second variable, this is equal to the measure $H(x,\frac{\sigma_N}{|\sigma_N|})|\sigma_N|d\mu = H(x,\frac{Du}{|Du|})|Du|$.

With the above notation, we define the relaxed set.
\begin{definition}[Relaxed set] \label{def:relaxed_space}
The set $\tilde{\s{K}}$ will be defined as the set of all pairs $(u,f) \in BV \times L^p$ such that $u \in L^\infty$, $u(T,\cdot) = u_T$ in the sense of traces, and
$-u_t + H(x,Du) \leq f$ in the sense of measures, by which we mean that $f - H(x,Du) + u_t$ is a non-negative Radon measure (using the notation given above).
\end{definition}
We remark that $\tilde{\s{K}}$ is a convex subset of the space $BV \times L^p$. The condition $u \in L^\infty$ appears extraneous until we invoke Lemma \ref{lem:integration_by_parts} in the arguments below, for which it is necessary that $u$ is sufficiently regular.

We redefine the Relaxed Problem \ref{pr:relaxed} as
\begin{problem}[New Relaxed Problem] \label{pr:new_relaxed}
Find $\inf_{(u,f) \in \tilde{\s{K}}} \s{A}(u,f)$.
\end{problem}
Note that the relaxed problem has a very clear connection with the smooth problem \ref{pr:smooth}, since smooth solutions to the Hamilton-Jacobi equation are also immediately in the relaxed set. However, the relationship to the original problem (involving viscosity solutions) is not so immediately clear. The relationship between distributional solutions and viscosity solutions is, in general, not well-understood. The only results in this direction of which the author is aware are for elliptic and parabolic equations, as in e.g. \cite{ishii95,juutinen01}. One may not expect in general for the concepts to be equivalent for hyperbolic equations. In our case, the infimum for the original problem \ref{pr:min_f} and for the relaxed problem \ref{pr:new_relaxed} are both equal to that of the smooth problem \ref{pr:smooth} because the ``integration by parts" formula (\ref{eq:integration_by_parts}) holds for viscosity solutions, even those which are not smooth. See Proposition \ref{prop:relaxed} below.

Before proceeding to the proof of existence of minimizers in this relaxed set, we prove two useful lemmas. The first is a helpful characterization of the statement that $-u_t + H(x,Du) \leq f$ for a function $u \in BV$.
\begin{lemma} \label{lem:vector_field_condition}
For $u \in BV$ and $f$ an integrable function, the following are equivalent:
\begin{itemize}
\item $-u_t + H(x,Du) \leq f$ in the sense of measures, and
\item for every continuous vector field ${\bf \tilde{v}} = {\bf \tilde{v}}(t,x) \in c(x,A)$ we have
\begin{equation} \label{eq:inequality_in_measure}
-\int_0^T \int_{\bb{T}^N} \phi(t,x)u_t(t,x) + \phi(t,x){\bf \tilde{v}}(t,x) \cdot Du(t,x) \leq \int_0^T \int_{\bb{T}^N} f(t,x)\phi(t,x)dt dx.
\end{equation}
for every continuous function $\phi : [0,T] \times \bb{T}^N \to [0,\infty)$.
\end{itemize}
\end{lemma}
\begin{proof}
One direction is easy:  given a continuous vecor field ${\bf \tilde{v}}(t,x) \in c(x,A)$ we have
\begin{equation}
-{\bf \tilde{v}}(t,x) \cdot \sigma_N(t,x) \leq H(x,\sigma_N(t,x)),
\end{equation}
so that if $-u_t + H(x,Du) \leq f$, then (\ref{eq:inequality_in_measure}) immediately follows.

For the other direction, let $\phi \geq 0$ be any continuous function. We want to construct a continuous vector field ${\bf \tilde{v}}(t,x) \in c(x,A)$ such that $-\iint \phi {\bf \tilde{v}} \cdot Du$ is arbitrarily close to $\iint \phi H(x,Du)$. We will be able to conclude from this that
\begin{equation} 
\int_0^T \int_{\bb{T}^N} \phi(t,x)(-u_t(t,x) + H(x,Du(t,x))) \leq \int_0^T \int_{\bb{T}^N} f(t,x)\phi(t,x)dt dx.
\end{equation}
Since $\phi$ is arbitrary, this implies $-u_t + H(x,Du) \leq f$ in the sense of measures.

Let $\epsilon > 0$. By Lusin's Theorem, there exists a compact set $K \subset [0,T] \times \bb{T}^N$ such that $|Du|([0,T] \times \bb{T}^N \setminus K) \leq \epsilon$ and the restriction $\nu|_K$ to $K$ of the function $\nu := \frac{Du}{|Du|}$ is continuous. Fix a $(t_0,x_0) \in K$. We can pick an optimal control $a_0 \in A$ such that $-c(x_0,a_0) \cdot \nu(t_0,x_0) = H(x_0,\nu(t_0,x_0))$. Since $(t,x) \mapsto -c(x,a) \cdot \nu(t,x)$ is (uniformly) continuous on $K$, uniformly in $a$, there exists a $\delta > 0$ (independent of $(t_0,x_0)$) such that $H(x,\nu(t,x)) \leq \epsilon-c(x,a_0) \cdot \nu(t,x)$ for all $(t,x) \in B_\delta(t_0,x_0) \cap K$.

Now cover $K$ with finitely many $\delta$-balls $B_\delta(t_i,x_i)$, $i= 1,\ldots,n$, and pick $a_i$ the optimal control corresponding to $(t_i,x_i)$. Let $\rho_1,\ldots,\rho_n$ be a partition of unity subordinate to this cover. Define
\begin{equation}
{\bf \tilde{v}}(t,x) = \sum_{i=1}^n \rho_i(t,x)c(x,a_i) = \sum_{(t,x) \in B_\delta(t_i,x_i)} \rho_i(t,x)c(x,a_i).
\end{equation}
Then ${\bf \tilde{v}}$ is continuous, and by the convexity of $c(x,A)$ for all $x$ we have that ${\bf \tilde{v}}(t,x) \in c(x,A)$. Moreover we get the estimate
\begin{align*}
-{\bf \tilde{v}}(t,x) \cdot \nu(t,x) &= -\sum_{(t,x) \in B_\delta(t_i,x_i)} \rho_i(t,x)c(x,a_i) \cdot \nu(t,x)\\
&\geq (H(x,\nu(t,x)) - \epsilon)\sum_{(t,x) \in B_\delta(t_i,x_i)} \rho_i(t,x)\\
&= H(x,\nu(t,x)) - \epsilon
\end{align*}
for all $(t,x) \in K$. We compute
\begin{align*}
\iint \phi(t,x)H(x,Du(t,x)) &= \iint_K \phi(t,x)H(x,\nu(t,x))|Du|(t,x) + \iint_{[0,T] \times \bb{T}^N \setminus K} \phi(t,x)H(x,\nu(t,x))|Du|(t,x)\\
&\leq \iint_K \phi(t,x)(\epsilon - {\bf \tilde{v}}(t,x) \cdot \nu(t,x))|Du|(t,x) + C\|\phi\|_\infty |Du|([0,T] \times \bb{T}^N \setminus K) \\
&\leq  -\iint_K \phi(t,x){\bf \tilde{v}}(t,x) \cdot Du(t,x) + \|\phi\|_\infty \epsilon |Du|(K) + C\|\phi\|_\infty \epsilon,
\end{align*}
where we have used the bounds on $H$ to obtain $H(x,\nu(t,x)) \leq C|\nu(t,x)| = C$. Similarly, since ${\bf \tilde{v}}$ is bounded it follows that
\begin{equation}
\left|\iint_{[0,T] \times \bb{T}^N \setminus K} \phi(t,x){\bf \tilde{v}}(t,x) \cdot Du(t,x) \right| \leq \|\phi\|_\infty \|{\bf \tilde{v}}\|_\infty |Du|([0,T] \times \bb{T}^N \setminus K) \leq \|\phi\|_\infty \|{\bf \tilde{v}}\|_\infty\epsilon.
\end{equation}
We conclude that for some constant $C > 0$ (not depending on $\epsilon$),
\begin{equation}
\int_0^T \int_{\bb{T}^N} \phi(t,x)H(x,Du(t,x)) \leq C\epsilon + \int_0^T \int_{\bb{T}^N} \phi(t,x){\bf \tilde{v}}(t,x) \cdot Du(t,x),
\end{equation}
which is what we needed to show. The proof is complete.
\end{proof}

The following lemma is key: it allows us to characterize members of the relaxed set in terms of the dual variable $m$ from Problem \ref{pr:dual}.

\begin{lemma}  \label{lem:integration_by_parts}
Suppose $u \in BV \cap L^\infty$ satisfies (\ref{eq:inequality_in_measure}) for every continuous vector field ${\bf \tilde{v}} = {\bf \tilde{v}}(t,x) \in c(x,A)$. Then for any $m \in L^q$ satisfying the continuity equation $m_t + \mathrm{div}(m{\bf v}) = 0$ for some vector field $\bf v$ with ${\bf v}(t,x) \in c(x,A)$ a.e., we have
\begin{align} 
\int_{\bb{T}^N} u(0,x)m(0,x)dx &\leq \int_{\bb{T}^N} u(t,x)m(t,x)dx + \int_0^t \int_{\bb{T}^N} f(s,x)m(s,x)dx ds, \nonumber \\
\int_{\bb{T}^N} u(t,x)m(t,x)dx &\leq \int_{\bb{T}^N} u_T(x)m(T,x)dx + \int_t^T \int_{\bb{T}^N} f(s,x)m(s,x)dx ds, \label{eq:trace_bound}
\end{align} 
 for almost every $t \in (0,T)$.
\end{lemma}

\begin{proof}
Set ${\bf w} = m{\bf v}$. For $\lambda \in (0,1]$ set $m_\lambda(t,x) = m(\lambda t,x)$ and ${\bf w}_\lambda(t,x) = \lambda {\bf w}(\lambda t,x)$ for $t \in [0,T],x \in \bb{T}^N$. We extend outside of $[0,T]$ be setting ${\bf w}_\lambda(t,x) = 0$ and extending $m_\lambda$ by a constant function. Observe that
\begin{equation}
\partial_t m_\lambda + \mathrm{div}~ {\bf w}_\lambda = 0
\end{equation}
holds in distributions for all $\lambda \in (0,1]$.

Let $\xi = \xi(t,x)$ be a standard convolution kernel, i.e. $\xi > 0$ smooth with $\int \int \xi = 1$. 
Set $\xi_\epsilon(t,x) = \epsilon^{-N-1}\xi((t,x)/\epsilon)$. Define $m_{\epsilon,\lambda} := \xi_\epsilon \ast m_\lambda$ and ${\bf w}_{\epsilon,\lambda} := \xi_\epsilon \ast {\bf w}_\lambda$. 
Let $\epsilon > 0$ be small and consider $\lambda \leq 1 - \epsilon L/c_0$, where $L$ is the Lipschitz constant of $c(\cdot,A)$ and $c_0 > 0$ is the size of the ball such that $B_{c_0}(0) \subset c(x,A)$ for all $x \in \bb{T}^N$. 
Then almost everywhere we have
\begin{align*}
{\bf w}_{\lambda}(s,y) &\in \lambda m_{\lambda}(s,y) c(y,A)\\
&\subset \lambda m_{\lambda}(s,y) (c(x,A) + (L\epsilon/c_0) B_{c_0}(0))\\
&\subset m_{\lambda}(s,y) (\lambda c(x,A) + (1 - \lambda) B_{c_0}(0))\\
&\subset m_{\lambda}(s,y) c(x,A)
\end{align*}
by the convexity of $c(x,A)$. 
Thus
\begin{multline}
{\bf w}_{\epsilon,\lambda}(t,x) = \int \int_{B_\epsilon(t,x)} \xi_\epsilon(t-s,x-y) {\bf w}_\lambda(s,y) ds dy \\ \in c(x,A) \int \int_{B_\epsilon(t,x)} \xi_\epsilon(t-s,x-y) {m}_\lambda(s,y) ds dy = m_{\epsilon,\lambda}(t,x)c(x,A).
\end{multline}
This result, together with the fact that ${\bf w}_{\epsilon,\lambda}$ and $m_{\epsilon,\lambda}$ are smooth and $m_{\epsilon,\lambda} > 0$, means that Equation (\ref{eq:inequality_in_measure}) applies with ${\bf \tilde{v}} = {\bf w}_{\epsilon,\lambda}/m_{\epsilon,\lambda}$. By integration by parts, this implies (\ref{eq:trace_bound}) with $m$ replaced by $m_{\epsilon,\lambda}$. On the other hand, $m_{\epsilon,\lambda} \to m$ in $L^q((0,T) \times \bb{T}^N)$ as $\epsilon \to 0, \lambda \to 1$, and in particular $m_{\epsilon,\lambda}(t) \to m(t)$ in $L^q(\bb{T}^N)$ for almost every $t \in (0,T)$. So passing to the limit on $\epsilon$ and $\lambda$ we finish the proof.
\end{proof}

An immediate application of this lemma is the following proposition, proving that the relaxed problem is indeed equivalent (has the same infimum) as the original problem.
\begin{proposition} \label{prop:relaxed}
The relaxed problem \ref{pr:new_relaxed} is in duality with \ref{pr:dual}, i.e. \begin{equation}
\inf_{(u,f) \in \tilde{\s{K}}} \s{A}(u,f) = - \min_{(m,{\bf w}) \in \s{K}_1} \s{B}(m,{\bf w}).
\end{equation}
Equivalently, the infimum appearing in Problem \ref{pr:new_relaxed} is the same as that appearing in Problem \ref{pr:min_f} and Problem \ref{pr:smooth}.
\end{proposition}

\begin{proof}
Since for $u \in \s{K}_0$ we have that $(u,-u_t + H(x,Du)) \in \tilde{\s{K}}$, it follows that $\inf_{(u,f) \in \tilde{\s{K}}} \s{A}(u,f) \leq \inf_{u \in \s{K}_0} \s{A}(u)$. It therefore suffices to prove the other direction. Equivalently, let $(m,{\bf w})$ be a minimizer of the dual problem \ref{pr:dual}; it suffices to show that for $(u,f) \in \tilde{\s{K}}$, we have $\s{A}(u,f) \geq -\s{B}(m,{\bf w})$. This essentially follows from condition (\ref{eq:trace_bound}). Indeed, we have
\begin{align*}
\s{A}(u,f) &= \int_0^T \int_{\bb{T}^N} K(t,x,f(t,x))dx dt - \int_{\bb{T}^N} u(0,x)m_0(x)dx\\
&\geq \int_0^T \int_{\bb{T}^N} f(t,x)m(t,x) - K^*(t,x,m(t,x))dx dt - \int_{\bb{T}^N} u(0,x)m_0(x)dx\\
&\geq -\int_0^T \int_{\bb{T}^N} K^*(t,x,m(t,x))dx dt - \int_{\bb{T}^N} u_T(x)m(T,x)dx = - \s{B}(m,{\bf w}).
\end{align*}
 
\end{proof}

The goal now is to show that Problem \ref{pr:new_relaxed} has a minimizer. We face two main difficulties. The first is the loss of continuity in the variable $f$. For an arbitrary $f \in L^p$ the integration of $f$ along individual trajectories in the proof of Lemma \ref{lem:ibp_inequality} is not well-defined. Fortunately, we may still expect Equation (\ref{eq:trace_bound}) to hold for $f \in L^p$ and $m \in L^q$ by duality. However, the second major difficulty we find is that $u$ need not be continuous for $(u,f) \in \tilde{\s{K}}$. In the following subsection we address precisely this phenomenon of ``loss of continuity" in passing to the limit on solutions of the Hamilton-Jacobi equation. After this we will show that nevertheless, a minimizer exists for Problem \ref{pr:new_relaxed}.

\subsubsection{Loss of regularity}
\label{sec:loss_of_continuity}

For Hamiltonians with super-linear growth one obtains H\"{o}lder continuity for viscosity solutions of the Hamilton-Jacobi equation (\ref{eq:hjb}), whose modulus of continuity does not depend on the continuity of $f$ appearing on the right-hand side but only on its $L^p$ norm \cite{cardaliaguet12}. This can be seen by proving an ``improvement of oscillation" result (\cite[Lemma 3.3]{cardaliaguet12}). The basic idea is that on successively smaller neighborhoods of an arbitrary point in space-time, the difference between the supremum and the infimum of a viscosity solution shrinks in proportion to the size of the neighborhood, which implies continuity. Showing this improvement of oscillation relies on the comparison principle and the classical Lax-Oleinik formula, which in the case where the Hamiltonian has super-linear growth provides upper and lower bounds on the point-wise behavior of solutions in terms of the values on parabolic cylinders. To be explicit, this means that the solution of the PDE
\begin{equation}
v_t + H(Dv) = 0 ~~ \text{in}~~ [-T,0] \times B_R
\end{equation}
is given by
\begin{equation} \label{eq:lax_oleinik}
v(t,x) = \min\{v(s,y) + (t-s)L\left(\frac{x-y}{t-s}\right)\}
\end{equation}
where the minimum is taken over the cylinder
$$ ([-T,0] \times \partial B_R) \cup (\{-T\} \times B_R).$$
Here $L$ is the Legendre transform of $H$.
When $H$ has super-linear growth the Lax-Oleinik formula is highly relevant, since its Legendre transform is finite at each point. In the case of linear growth, for instance $H(Du) = |Du|$, the Legendre transform is simply an indicator function and therefore infinite outside a compact set. Indeed, in this case the minimum in formula \ref{eq:lax_oleinik} can be taken over a light cone, excluding those parts of the cylinder where $L$ is infinite.

For this reason, the improvement of oscillation argument will not work, and any degree of continuity of solutions to (\ref{eq:back_hj}) must correspond to the continuity of $f$. However, we can at least obtain the following result:
\begin{lemma} \label{lem:holder} Suppose $u$ is a continuous viscosity solution to the Hamilton-Jacobi equation
\begin{equation}
-u_t(t,x) + H(x,Du) = f(t,x)
\end{equation}
in a region $[0,T] \times U$ for some open set $U$ in Euclidean space.
We take the assumptions from Section \ref{sec:assumptions} to be given. Assume also that $f$ is continuous with $f \geq 0$. Then for any $0 \leq t < s \leq T$ and for any $0 < \beta < 1$, we have the following estimate: 
\begin{equation} \label{eq:holder_step1}
u(t,x) - u(s,y) \leq C (1 - \beta^2)^{-N/2p}\|f\|_p |t-s|^{\alpha}, ~~~ \forall |x-y| \leq \beta c_0(s-t),
\end{equation}
where
$\alpha = 1-(N+1)/p$
and the constant $C$ depends on $p,N$, and $c_0^{-1}$.
\end{lemma}

\begin{remark}
In fact, by the comparison principle, it is only necessary that $u$ be a subsolution to the HJ equation in Lemma \ref{lem:holder}.
\end{remark}

\begin{proof}
The main idea is to use trajectories of the ODE (\ref{eq:open_loop1}) as ``subcharacteristics," i.e. curves along which the solution is monotone (not necessarily constant). This hinges on the dynamic programming principle stated in Equation (\ref{eq:dpp}) combined with the assumption that the ball $B_{c_0}$ is contained in the set of admissible velocities $c(x,A)$ for all $x$.

Suppose $|x-y| \leq \beta c_0(s-t)$, where $0 < \beta < 1$. Let
\begin{equation}
\theta := \frac{x-y}{s-t}, ~~ R_\theta := B(0;\sqrt{c_0^2 - |\theta|^2}) \subset B(\theta;c_0) \cap B(-\theta;c_0).
\end{equation}
We note that $R_\theta$ is a nonempty region because $|\theta| < c_0$. Now for each $\sigma \in R_\theta$, define
\begin{equation}
\gamma_\sigma(\tau) = x + \theta(\tau-t) + \left\{\begin{array}{ccc}
\sigma(\tau-t) & \text{if} & t \leq \tau \leq \frac{t+s}{2}\\
\sigma(s-\tau) & \text{if} & \frac{t+s}{2} \leq \tau \leq s
\end{array} \right.
\end{equation}
which we note goes from $x$ at time $t$ to $y$ at time $s$. Moreover, we have that
\begin{equation}
|\dot{\gamma}_\sigma(\tau)| = |\theta \pm \sigma| \leq c_0,
\end{equation}
and therefore $\dot{\gamma}_\sigma(\tau) \in c(\gamma(\tau),A)$. In other words, $\gamma_\sigma$ is an admissible trajectory of the ODE (\ref{eq:open_loop1}).
Then Equation (\ref{eq:dpp}) applies to $\gamma_\sigma$ for all $\sigma \in R_\theta$, hence
\begin{equation} \label{eq:holder_step1-1}
u(t,x) - u(s,y) \leq \frac{1}{|R_\theta|}\int_{R_\theta}\int_t^s f(\tau,\gamma_\sigma(\tau))d\tau d\sigma.
\end{equation}
The integral on the right-hand side is estimated in two parts. We estimate the first part, then note that the second part is handled in the same way. The first part is
\begin{align*}
\int_{R_\theta}\int_t^{\frac{t+s}{2}} & f(\tau, x + (\sigma+\theta)(\tau - t))d\tau d\sigma\\
&= \int_t^{\frac{t+s}{2}}\int_{\tilde{R}_\theta} f(\tau,\eta)(\tau-t)^{-(N+1)} d\eta d\tau\\
&\leq \left(\int_t^{\frac{t+s}{2}}\int_{\tilde{R}_\theta} |f(\tau,\eta)|^p d\eta d\tau\right)^{1/p}\left(\int_s^{\frac{t+s}{2}}\int_{\tilde{R}_\theta} (\tau-t)^{-(N+1)p^*} d\eta d\tau\right)^{1/p^*}\\
&\leq \left(\frac{s-t}{2}\right)^{1/p}\|f\|_p \left(\int_0^{\frac{s-t}{2}}|R_\theta| \rho^{(1-p^*)(N+1)} d\rho \right)^{1/p^*}\\
&\leq C(p,N)|R_\theta|^{1/p^*}\|f\|_p (s-t)^{1-(N+1)/p},
\end{align*}
where we have used a series of change of variables and H\"{o}lder's inequality. (In particular, for each $\tau$ the region $\tilde{R}_\theta$ is obtained from $R_\theta$ by the transformation $\sigma \mapsto x + (\sigma + \theta)(\tau - t)$. Thus the volume scales by $(\tau -s)^{(N+1)}$.) We arrive at the estimate
\begin{equation}
u(t,x) - u(s,y) \leq C(p,N)|R_\theta|^{-1/p}\|f\|_p |s-t|^{1-(N+1)/p}.
\end{equation}
Then considering that
\begin{equation}
|R_\theta| = C_N(c_0^2 - |\theta|^2)^{N/2} \geq C_N(1-\beta^2)^{N/2}c_0^N,
\end{equation}
we obtain (\ref{eq:holder_step1}) as desired, where the constant $C(p,N)$ may have changed but remains dependent only on $p$ and $N$.

\end{proof}

Unfortunately, this lemma implies no more than right-hand continuity in time, and in practice we will use it simply to obtain an upper bound on $u$ which depends only on the $L^p$ norm of $f$. Indeed, since the estimate (\ref{eq:holder_step1}) is restricted to points in a light cone, it fails to imply any sort of continuity in space for the function $u(t,\cdot)$. What is happening is that in this case, where the Hamiltonian has linear growth rather than super-linear, the equation (\ref{eq:hjb}) has a more strictly hyperbolic structure, and it is impossible to get any sort of regularization of irregular data.

We illustrate this with the following counterexample, showing that a sequence of continuous functions which converge in $L^p$ (in this case $L^\infty$) to some discontinuous function $f$ induces a sequence of value functions which become discontinuous in the limit.
Let $f_\epsilon$ be a continuous function taking values in $[0,1]$ such that, for $(t,x) \in [0,1] \times [0,2]$ we have
\begin{equation}
f_\epsilon(t,x) = \left\{ \begin{array}{ccc} 1 & \text{if} & 1 - t \leq x \leq 1 + t \\ 0 & \text{if} & x \leq 1 - t - \epsilon ~\text{or}~ x \geq 1 + t + \epsilon \end{array} \right.
\end{equation}
Note that $f_\epsilon \to f_0 =: f$ as $\epsilon \to 0$. We can think of $f$ as an obstacle that one begins to construct at time zero and then grows out in both directions at speed one.

We will take $c \equiv 1$ and $u_T = u_1 \equiv 0$. Consider the value function $u_\epsilon$ given by (\ref{eq:value}) with $f$ replaced by $f_\epsilon$, that is,
\begin{equation}
u_\epsilon(t,x) = \inf\{\int_t^1 f(s,y(s))ds : y(t) = x, |\dot{y}| \leq 1 ~ \text{a.e.}\}.
\end{equation}
If $1-t \leq x \leq 1 + t$, then for any $y(\cdot)$ such that $y(t) = x$ and $|\dot{y}| \leq 1$, then $1-s \leq y(s) \leq 1+s$ for all $s \in [t,1]$ and therefore $f(s,y(s)) = 1$. On the other hand, if $x \leq 1 - t - \epsilon$ then take $y(s) = x + t -s$ to get a trajectory with $|\dot{y}| = 1$ and $y(s) \leq 1 - s - \epsilon$, and therefore $f(s,y(s)) = 0$ for all $s \in [t,1]$. Similarly, if $x \geq 1 + t + \epsilon$ then take $y(s) = x - t + s$ to get a trajectory with $|\dot{y}| = 1$ and $y(s) \geq 1 + s + \epsilon$, and therefore $f(s,y(s)) = 0$ for all $s \in [t,1]$. It follows that
\begin{equation}
u_\epsilon(t,x) = \left\{ \begin{array}{ccc} 1-t & \text{if} & 1 - t \leq x \leq 1 + t \\ 0 & \text{if} & x \leq 1 - t - \epsilon ~\text{or}~ x \geq 1 + t + \epsilon \end{array} \right.
\end{equation}
and thus $u_\epsilon$ converges pointwise to a function $u$ given by
\begin{equation}
u(t,x) = \left\{ \begin{array}{ccc} 1-t & \text{if} & 1 - t \leq x \leq 1 + t \\ 0 & \text{if} & x < 1 - t ~\text{or}~ x > 1 + t \end{array} \right. = (1-t)\chi_{|x-1| \leq t}(t,x)
\end{equation}
Observe that $u$ is not continuous. This ensures its derivative is not (Lebesgue) integrable. Indeed, one computes
\begin{equation}
u_x(t,\cdot) = (1-t)(\delta_{1-t} - \delta_{1+t})
\end{equation}
in distribution, so that $u_x(t)$ is a discrete measure.

On the other hand, from the point of view of controlling front propagation, the above example shows explicitly how the obstacle $f$ effectively ``blocks" the front in finite time (without adressing optimality). Indeed, if the sub-level set $\{u(t,\cdot) \leq 0\}$ is the set which is ``on fire" at time $t$, then in the above example this is precisely the set $[0,1-t) \cup (1+t,2]$, which is empty at time $t = 1$.

\subsubsection{Existence of minimizers of the relaxed problem}
\label{sec:existence_relaxed}

Let us now prove the existence of minimizers to the relaxed problem.
\begin{theorem} \label{thm:existence}
The Relaxed Problem \ref{pr:new_relaxed} has at least one solution $(u,f) \in \tilde{\s{K}}$, which also satisfies $f \geq 0$.
\end{theorem}

\begin{proof}

Let $f_n$ be a minimizing sequence of continuous functions for Problem \ref{pr:min_f}, and let $u_n$ be the unique viscosity solution to (\ref{eq:back_hj}) with $f$ replaced by $f_n$, so that in particular $u_n$ is given by the value function (\ref{eq:value}). We can assume without loss of generality that $f_n \geq 0$. Otherwise, let $\tilde{f}_n := \max\{0,f_n\}$ and let $\tilde{u}_n$ be the corresponding viscosity solution to (\ref{eq:back_hj}). Then by the comparison principle we have $\tilde{u}_n \geq u_n$ so that $-\int \tilde{u}_n(0)m_0 \leq -\int u_n(0)m_0$ while $|\tilde{f}_n| \leq |f_n|$ implies $K(t,x,\tilde{f}_n) \leq K(t,x,f_n)$ by the assumptions on $K$. It follows that $\s{J}(\tilde{f}_n) \leq \s{J}(f_n)$ and so $\tilde{f}_n$ is also a minimizing sequence.

In fact, without loss of generality, we can also assume that $f_n$ is Lipschitz. To see this will require a bit more explanation. Let $f$ be any continuous function on $[0,T] \times \bb{T}^N$ and let $u$ be the corresponding viscosity solution to (\ref{eq:back_hj}). Let $\xi \in C^\infty$ be a standard convolution kernel, with support in $\overline{B_1}$ and $\int \xi = 1$, and let $\xi_\epsilon(t,x) = \epsilon^{-N-1}\xi((t,x)/\epsilon)$ be the corresponding standard mollifiers. Since $f$ is continuous on a compact space, it is uniformly continuous, so let $\omega_f$ denote its modulus. Define $f_\epsilon := \xi_\epsilon \ast f$, and note that this is a smooth function (in particular, it is Lipschitz). Then $\|f_\epsilon - f\|_\infty \leq \omega_f(\epsilon)$ by definition of convolution. Denote by $u_\epsilon$ the viscosity solution of (\ref{eq:back_hj}) corresponding to $f_\epsilon$. By direct inspection of the formula (\ref{eq:value}), we find that
\begin{equation}
u_\epsilon(0,x) \geq -\omega_f(\epsilon)T + u(0,x).
\end{equation}
We can say more than this, but the information above suffices to conclude that $\limsup_{\epsilon \to 0} \s{J}(f_\epsilon) \leq \s{J}(f)$. It therefore suffices to minimize over smooth functions $f$. For our purposes it is enough that each function $f_n$ is Lipschitz, since then by standard results (see, for instance, \cite[Proposition 3.1]{bardi97}) we have that $u_n$ is Lipschitz. Thus, by the Rademacher differentiation theorem, $u_n$ is differentiable almost everywhere.

Now from Lemma \ref{lem:holder} we have that
\begin{equation} \label{eq:est_above}
u_n(t,x) \leq u_T(x) + C(T-t)^\alpha \|f_n\|_p,
\end{equation}
where $\alpha = 1 - (N+1)/p$.
From that fact that $u_n$ is a minimizing sequence, we have that
\begin{align*}
C_1 &\geq \int_0^T \int_{\bb{T}^N} K(t,x,f_n(t,x)) dx dt - \int_{\bb{T}^N} u_n(0,x) m_0(x)dx\\
&\geq \frac{1}{C}\|f_n\|_p^p - C - \|u_T\|_\infty \|m_0\|_1 - CT^\alpha \|f_n\|_p.
\end{align*}
It follows that $f_n$ is bounded in $L^p$ and that, from Equation (\ref{eq:est_above}), $u_n$ is bounded pointwise from above.

Moreover, we can get a lower point-wise bound for $u_n$ by setting a function $w$ to be the viscosity solution of (\ref{eq:back_hj}) with $f$ replaced by the zero function. Then $w$ inherits the Lipschitz regularity of $u_T$ and we have
\begin{equation}
\label{eq:est_below}
u_T(x) - C(T-t) \leq w(t,x) \leq u_n(t,x)
\end{equation}
by the comparison principle.

Based on equation (\ref{eq:back_hj}) and (\ref{eq:est_above}) and using the fact that $u_n$ is almost everywhere differentiable, we have
\begin{align*}
\int_0^T \int_{\bb{T}^N} H(x,Du_n(t,x)) dx dt &= \int_0^T\int_{\bb{T}^N} \partial_t u_n(t,x) + f_n(t,x) dxdt \\
&\leq \int_{\bb{T}^N}u_T(x) - u_n(0,x)dx + \|f_n\|_1\\
&\leq C(T^\alpha + 1)\|f_n\|_p.
\end{align*}
By Equation (\ref{eq:hamiltonian_bounds}), this implies that $Du_n$ and thus also $H(x,Du_n)$ are bounded in $L^1(0,T;L^1(\bb{T}^N))$. Likewise, $\partial_t u_n = H(x,Du_n) - f_n$ is bounded in $L^1(0,T;L^1(\bb{T}^N))$ and so the full gradient $(\partial_t u_n,Du_n)$ is uniformly bounded in $L^1$. Since $u_n$ is also uniformly bounded in $L^\infty$, we have that $u_n$ is uniformly bounded in $BV$. It follows that
\begin{itemize}
\item $u_n \rightharpoonup u$ in the weak$^*$ topology on $BV$,
\item $u_n \to u$ in $L^1$ and hence almost everywhere, and
\item $f_n \rightharpoonup f$ weakly in $L^p$ for some function $f$.
\end{itemize}
Note that $u \in L^\infty$ by virtue of the uniform bounds on $u_n$. We also have, by applying at the same time Equations (\ref{eq:est_above}) and (\ref{eq:est_below}) and passing to the limit, that $u_n(T,x) = u_T(x)$ for all $n$ and for all $x$ and therefore (by a standard argument) that $u(T,\cdot) = u_T$ in the sense of traces.

To see that $-u_t + H(x,Du) \leq f$ in the sense of measures, let us show that (\ref{eq:inequality_in_measure}) holds and apply Lemma \ref{lem:vector_field_condition}. (Alternatively, one could get a direct proof by carefully applying Reshetnyak's Theorem \cite{reshetnyak68}.) Let ${\bf \tilde{v}} = {\bf \tilde{v}}(t,x) \in c(x,A)$ be a continuous vector field and let $\phi : [0,T] \times \bb{T}^N \to [0,\infty)$ be a continuous function. Then
\begin{equation} \label{eq:inequality_in_measure2}
-\int_0^T \int_{\bb{T}^N} \phi(t,x)\partial_t u_n(t,x) + \phi(t,x){\bf \tilde{v}}(t,x) \cdot Du_n(t,x)dt dx \leq \int_0^T \int_{\bb{T}^N} f_n(t,x)\phi(t,x)dt dx
\end{equation}
for each $n$ directly from the Hamilton-Jacobi equation $-\partial_t u_n + H(x,Du_n) = f_n$. Since $u_n \rightharpoonup u$ in the weak$^*$ topology on $BV$ and $f_n \rightharpoonup f$ weakly in $L^p$, we deduce (\ref{eq:inequality_in_measure}) by passing to the limit in (\ref{eq:inequality_in_measure2}). 

We have thus shown that $(u,f) \in \s{K}$. We also have, by convexity, that
\begin{multline}
\int_0^T \int_{\bb{T}^N} K(t,x,f(t,x)) dx dt - \int_{\bb{T}^N} u(0,x) dm_0(x)\\ \leq \liminf_{n \to \infty} \int_0^T \int_{\bb{T}^N} K(t,x,f_n(t,x)) dx dt - \int_{\bb{T}^N} u_n(0,x) dm_0(x)
\end{multline}
and $(u,f)$ is therefore a minimizer. Finally, the condition $f \geq 0$ follows because, without loss of generality, $f_n \geq 0$ for all $n$.

\end{proof}

\section{Characterization of minimizers}

\label{sec:characterization}

In this section our goal is to characterize the minimizers of Problem \ref{pr:min_f}. 
Heuristically, one may use the Lagrange multiplier method for the constrained optimization problem to characterize the minimizers by the following system of PDEs.
\begin{equation} \label{eq:mfg}
\begin{array}{rcll}
-u_t + H(x,Du) &=& k(t,x,m(t,x))  & \text{(state equation)}\\
u(T,x) &=& u_T(x) & \\
m_t - \mathrm{div} \left(m\partial_p H(x,Du)\right) &=& 0 & \text{(adjoint equation)}\\
m(0,x) &=& m_0(x).
\end{array}
\end{equation}
System (\ref{eq:mfg}) is in fact a model for mean field games with local coupling \cite{cardaliaguet13}, see also \cite{lasry06,lasry06a,lasry07}.

The goal of this section is to show that minimizers of Problem \ref{pr:min_f} correspond to weak solutions of the PDE system (\ref{eq:mfg}). We first set about the task of defining weak solutions, and then we prove an existence and uniqueness theorem.

\subsection{Definition of weak solutions}
\label{sec:definition_weak}

In order to make sense of the PDE system (\ref{eq:mfg}), we first need to define weak solutions, taking into account that the existence of smooth, pointwise solutions may be unrealistic. In particular, the vector field defined by $\partial_p H(x,Du)$ may not be well-defined even for smooth functions $u$, owing to the lack of differentiability of $H$. Our definition of solution is inspired by the optimization problems \ref{pr:new_relaxed} and \ref{pr:dual}. In the next section, we should that weak solutions are essentially equivalent to minimizers of these problems.

\begin{definition} \label{def:weak}
A pair $(u,m) \in BV((0,T) \times \bb{T}^N) \times L^q((0,T) \times \bb{T}^N)$ is called a weak solution to the system (\ref{eq:mfg}) if it satisfies the following conditions.
\begin{enumerate}
\item $u$ satisfies  a first-order Hamilton-Jacobi equation
\begin{equation} \label{eq:hjb_weak}
-u_t + H(x,Du) = k(t,x,m)
\end{equation}
in the following weak sense:
\begin{align} 
\int_{\bb{T}^N} u(0,x)m_0(x)dx &= \int_{\bb{T}^N} u(t,x)m(t,x)dx + \int_0^t \int_{\bb{T}^N} k(s,x,m(s,x))m(s,x)dx ds, \nonumber \\
\int_{\bb{T}^N} u(t,x)m(t,x)dx &= \int_{\bb{T}^N} u_T(x)m(T,x)dx + \int_t^T \int_{\bb{T}^N} k(s,x,m(s,x))m(s,x)dx ds, \label{eq:ibp_weak}
\end{align}
for almost all $t \in [0,T]$. We also have ``$-u_t + H(x,Du) \leq k(t,x,m)$ in the sense of measures," in the sense of inequality (\ref{eq:inequality_in_measure}).

Moreover, $u(T,\cdot) = u_T$ in the sense of traces
\item $m$ satisfies the continuity equation
\begin{equation} \label{eq:continuity_weak}
m_t + \mathrm{div}~(m{\bf v}) = 0 ~~~\text{in}~~(0,T) \times \bb{T}^N, ~~~ m(0) = m_0
\end{equation}
in the sense of distributions, where ${\bf v} \in L^\infty((0,T) \times \bb{T}^N;\bb{R}^N)$ is a vector field such that ${\bf v}(t,x) \in c(x,A)$ a.e.
\end{enumerate}
\end{definition}
A few remarks on this solution: first, we mention that (\ref{eq:hjb_weak}) with its weak formulation (\ref{eq:ibp_weak}) corresponds to the state equation in (\ref{eq:mfg}). It is (unfortunately) natural that the equation not hold in an almost everywhere sense or even in a viscosity sense, due to the very low regularity of $u$ and $m$. On the other hand, Equation (\ref{eq:continuity_weak}) is the natural interpretation of the adjoint equation in (\ref{eq:mfg}), since $\bf v$ here is taken heuristically as the optimal feedback for the control problem \ref{eq:value}. This heuristic is justified by the result of Theorem \ref{thm:minimizers_weak} which implies that ${\bf v} \cdot Du = -H(x,Du)$ on the set $\{m > 0\}$, provided that the derivative of $u$ is Lebesgue integrable. In general, however, very little can be said to definitively characterize $\bf v$. Finally, we note that (\ref{eq:ibp_weak}) could be deduced from (\ref{eq:continuity_weak}) by integration by parts if $u$ and $m$ were smooth. In the weak setting, we will see that these equations are ultimately justified using Equation (\ref{eq:trace_bound}) established for minimizers of Problem \ref{pr:new_relaxed}.

\subsection{Existence and uniqueness of weak solutions}
\label{sec:existence_weak}

The main result of this section is the following.
\begin{theorem}[Existence and partial uniqueness of weak solutions] \label{thm:minimizers_weak}
(i) If $(m,{\bf w}) \in \s{K}_1$ is a minimizer of \ref{eq:dual} and $(u,f) \in \tilde{\s{K}}$ is a minimizer of \ref{pr:new_relaxed}, then $(u,m)$ is a weak solution of (\ref{eq:mfg}) and $f(t,x) = k(t,x,m(t,x))$ almost everywhere. Moreover, if $u \in W^{1,p}$, then the Hamilton-Jacobi equation
\begin{equation} \label{eq:pointwise_HJ}
-u_t - {\bf v} \cdot Du = f
\end{equation}
holds pointwise almost everywhere in $\{m > 0\}$, where $\bf v$ is a bounded vector field given by $m{\bf v} = {\bf w}$.

(ii) Conversely, if $(u,m)$ is a weak solution of (\ref{eq:mfg}), then there exist functions ${\bf w}, f$ such that $(u,f) \in \tilde{\s{K}}$ is a minimizer of \ref{pr:new_relaxed} and $(m,{\bf w}) \in \s{K}_1$ is a minimizer of \ref{pr:dual}.

(iii) If $(u,m)$ and $(u',m')$ are both weak solutions to (\ref{eq:back_hj}), then $m = m'$ almost everywhere while $u = u'$ almost everywhere in the set $\{m > 0\}$.
\end{theorem}

Before turning to the proof, we make a couple of remarks.
Concerning the condition (\ref{eq:pointwise_HJ}) almost everywhere on the set $\{m > 0\}$, we remark that in principle, this equation justifies the unrigorous statement ${\bf v} = \partial_pH(x,Du)$, but here we have to assume some integrability of the derivative of $u$. We have already seen an example in which this is not the case (Section \ref{sec:loss_of_continuity}). Thus in general the vector field $\bf v$ has a rather vague relationship with $u$, and we have so far not been able to prove a more precise result.

Concerning the uniqueness of solutions, the weak solution to the Hamilton-Jacobi equation is unique only on the set $\{m > 0\}$ ($m$ is unique). This is due to the very low regularity of the solutions. If we had continuity of solutions $u$ as well as integrability of the Hamiltonian term $H(x,Du)$, then it would be possible to obtain pointwise uniqueness by appealing to the comparison principle for viscosity solutions. Essentially, we would have $-u_t + H(x,Du) = 0$ in a viscosity sense on any open set in $\{m=0\}$, from which uniqueness follows. See the corresponding uniqueness argument in \cite{cardaliaguet13}.

Finally, we mention an important corollary of the theorem:
\begin{corollary}[Approximation of minimizers]
\label{cor:approximation}
If $(u,f) \in \tilde{\s{K}}$ is a minimizer of Problem \ref{pr:new_relaxed}, then there exist continuous (even Lipschitz) functions $(u_n,f_n)$ such
\begin{itemize}
\item $u_n$ is the unique viscosity solution to (\ref{eq:back_hj}) with $f$ replaced by $f_n$,
\item $u_n \to u$ almost everywhere on $\{m > 0\}$ and $f_n \rightharpoonup f$ weakly in $L^p((0,T) \times \bb{T}^N)$, where $m$ is the unique $L^q$ function such that $(m,{\bf w}) \in \s{K}_1$ is a minimizer of Problem \ref{pr:dual} for some ${\bf w}$.
\end{itemize}
\end{corollary}
Note, however, that we do not have a method to explicitly construct approximations of solutions.

We now turn to the proof of Theorem \ref{thm:minimizers_weak}.

\begin{proof}[Proof of Theorem \ref{thm:minimizers_weak}]
(i) Suppose $(m,{\bf w}) \in \s{K}_1$ is a minimizer of \ref{eq:dual} and $(u,f) \in \tilde{\s{K}}$ is a minimizer of \ref{pr:new_relaxed}.
By Proposition \ref{prop:relaxed} and condition (\ref{eq:trace_bound}) we have
\begin{align}
0 &= \int_0^T \int_{\bb{T}^N} K(t,x,f(t,x)) + K^*(t,x,m(t,x)) dxdt + \int_{\bb{T}^N} u_T(x)m(T,x) - u(0,x)m_0(x)dx \nonumber \\
&\geq \int_0^T \int_{\bb{T}^N} f(t,x)m(t,x) dxdt + \int_{\bb{T}^N} u_T(x)m(T,x) - u(0,x)m_0(x)dx \geq 0. \label{eq:fm_equals}
\end{align}
It follows that all the inequalities above are equalities. We deduce that
\begin{equation}
K(t,x,f(t,x)) + K^*(t,x,m(t,x)) = f(t,x)m(t,x)
\end{equation}
almost everywhere. Therefore $f = k(t,x,m(t,x))$ almost everywhere by definition of Fenchel conjugate. Note that $m \geq 0$ and thus by the assumptions on $k$ we have $f \geq 0$ and $\{f > 0\} = \{m > 0\}$. Applying Equation (\ref{eq:fm_equals}) to Equation (\ref{eq:trace_bound}) yields (\ref{eq:ibp_weak}).

Next, from the fact that $(m,{\bf w}) \in \s{K}_1$ we have the characterization
\begin{align*}
m_t + \mathrm{div}~({\bf w}) &= 0,\\
m(T,\cdot) &= m_0(\cdot),\\
{\bf w}(t,x) &\in m(t,x)c(x,A).
\end{align*}
Letting ${\bf v} = {\bf w}/m$ on the set $m > 0$ and ${\bf v} = 0$ on the set $m = 0$, we see that $\bf v$ is a vector field in $L^\infty$ with ${\bf v}(t,x) \in c(x,A)$ almost everywhere and such that $m,{\bf v}$ satisfies the continuity equation (\ref{eq:continuity_weak}). We conclude that $(u,m)$ is a weak solution of (\ref{eq:mfg}).

In the case when $u \in W^{1,p}$, we can show in a straightforward manner that (\ref{eq:pointwise_HJ}) holds almost everywhere on $\{m > 0\}$. We appeal to the continuity equation and then to (\ref{eq:ibp_weak}):
\begin{equation}
\int_0^T \int_{\bb{T}^N} {\bf w} \cdot Du = \int_0^T \int_{\bb{T}^N} m(T)u_T - m_0u(0) - \int_0^T \int_{\bb{T}^N} mu_t = - \int_0^T \int_{\bb{T}^N} m(f+u_t).
\end{equation}
It follows that $m(f+u_t+{\bf v} \cdot Du) = 0$ almost everywhere, as desired.

Conversely, suppose $(u,m)$ is a weak solution of (\ref{eq:mfg}). We set ${\bf w} = m{\bf v}$ and $f = k(\cdot,\cdot,m)$. Since $m \in L^q$ it follows that $f \in L^p$ by the growth conditions on $k$ in (\ref{eq:cost_growth3}). It then follows from the definition of weak solution that $(u,f) \in \tilde{\s{K}}$; in particular, Equation (\ref{eq:trace_bound}) comes from Equation (\ref{eq:ibp_weak}). Likewise, since basic results on transport theory (c.f. \cite{ambrosio08}) imply that $m \geq 0$, the definition of weak solution gives $(m,{\bf w}) \in \s{K}_1$. We want to then show that $(u,f)$ solves (\ref{pr:new_relaxed}) and that $(m,{\bf w})$ solves (\ref{pr:dual}).

We begin with $(u,f)$. Take another pair $(u',f') \in \s{K}$. By convexity of $K$ in the third variable, we have
\begin{align*}
\s{A}(u',f') &= \int_0^T \int_{\bb{T}^N} K(t,x,f'(t,x)) dx dt - \int_{\bb{T}^N} u'(0) m_0 ~ dx\\
&\geq \int_0^T \int_{\bb{T}^N} K(t,x,f(t,x)) + \partial_f K(t,x,f)(f'-f) dx dt - \int_{\bb{T}^N} u'(0) m_0 ~ dx\\
&= \int_0^T \int_{\bb{T}^N} K(t,x,f(t,x)) + m(f'-f) dx dt - \int_{\bb{T}^N} u'(0) m_0 ~ dx
\end{align*}
using the fact that $f$ and $m$ are conjugate.
We have from (\ref{eq:hjb_weak}) and (\ref{eq:ibp_weak}) in the definition of weak solution that
\begin{equation}
-\int_0^T \int_{\bb{T}^N} fm = \int_{\bb{T}^N} m(T)u_T - m_0u(0).
\end{equation}
On the other hand, Equation (\ref{eq:trace_bound}) in the definition of $\tilde{\s{K}}$ implies
\begin{equation}
\int_0^T \int_{\bb{T}^N} f'm \geq \int_{\bb{T}^N} -m(T)u_T + m_0u'(0).
\end{equation}
Adding these together we have that
\begin{equation*}
\s{A}(u',f') \geq \int_0^T \int_{\bb{T}^N} K(t,x,f(t,x))  dx dt - \int_{\bb{T}^N} u(0) m_0 ~ dx = \s{A}(u,f),
\end{equation*}
and thus $(u,f)$ is a minimizer of the relaxed problem.

We now turn to $(m,{\bf w})$, for which the proof is similar.  Let $(m',{\bf w}')$ be any minimizer of $\s{B}$ in $\s{K}_1$. Then by the assumptions $K^*$ we have
\begin{align*}
\s{B}(m',{\bf w}') &= \int_{\bb{T}^N} u_T m'(T)dx + \int_0^T \int_{\bb{T}^N} K^*(t,x,m'(t,x)) dx dt\\
&\geq \int_{\bb{T}^N} u_Tm'(T)dx + \int_0^T \int_{\bb{T}^N} K^*(t,x,m(t,x)) + k(t,x,m)(m'-m) dx dt\\
&= \int_{\bb{T}^N} u_Tm'(T)dx + \int_0^T \int_{\bb{T}^N} K^*(t,x,m(t,x)) + f(m'-m) dx dt\\
&= \int_{\bb{T}^N} u_Tm'(T) + u_Tm(T) - m_0u(0)dx + \int_0^T \int_{\bb{T}^N} K^*(t,x,m(t,x)) + fm' dx dt.
\end{align*}
All that remains is to observe that
\begin{equation}
\int_0^T \int_{\bb{T}^N} fm' \geq \int_{\bb{T}^N} -m'(T)u_T + m_0u(0),
\end{equation}
which follows from Equation (\ref{eq:trace_bound}) and the fact that $(u,f)$ is an element of $\tilde{\s{K}}$. From this inequality we find
\begin{equation*}
\s{B}(m',{\bf w}') \geq \int_{\bb{T}^N} u_Tm(T)dx + \int_0^T \int_{\bb{T}^N} K(t,x,m(t,x)) dx dt = \s{B}(m,{\bf w}),
\end{equation*}
hence $(m,{\bf w})$ is a minimizer for the dual problem.

(iii) Suppose $(u_1,m_1)$ and $(u_2,m_2)$ are both weak solutions to (\ref{eq:back_hj}). Since $m_1$ and $m_2$ are both minimizers of Problem \ref{pr:dual} by Part (ii) and this problem is strictly convex in the variable $m$, it follows that $m_1 = m_2 =: m$ almost everywhere. We let $f = k(\cdot,\cdot,m)$ and note that by Part (ii) both $(u_1,f)$ and $(u_2,f)$ are minimizers of Problem \ref{pr:new_relaxed}. Let $u := \max\{u_1,u_2\}$. Below we will prove that $u$ satisfies (\ref{eq:inequality_in_measure}). This implies that $(u,f) \in \tilde{\s{K}}$, and moreover $(u,f)$ is also a minimizer of Problem \ref{pr:dual} since $- \int u(0)m_0 \leq - \int u_1(0)m_0$. Let us see that this implies $u_1 = u = u_2$ on $\{m > 0\}$. Indeed, since both $(u_1,f)$ and $(u,f)$ are minimizers, we have
\begin{equation}
\int_0^T \int_{\bb{T}^N} K(t,x,f(t,x)) dx dt - \int_{\bb{T}^N} u_1(0,x) m_0(x)dx = \int_0^T \int_{\bb{T}^N} K(t,x,f(t,x)) dx dt - \int_{\bb{T}^N} u(0,x) m_0(x)dx
\end{equation}
so that
\begin{equation}
\int_{\bb{T}^N} (u(0,x) - u_1(0,x)) m_0(x)dx = 0.
\end{equation}
Since the integrand here is non-negative, it follows that $u(0) = u_1(0)$ almost everywhere on $\{m_0 > 0\}$. By the same token, Equation (\ref{eq:ibp_weak}) implies that $u(t,x) = u_1(t,x)$ almost everywhere in $\{m(t,\cdot) > 0\}$, for almost every $t \in (0,T)$. We conclude that $u = u_1$ almost everywhere in the set $\{m > 0\}$. The same argument holds with $u_1$ replaced by $u_2$, so $u_1 = u_2$ almost everywhere in $\{m > 0\}$.

It remains only to show that $u$ is a subsolution of the Hamilton-Jacobi equation, that is to say it remains to show (\ref{eq:inequality_in_measure}) holds. A key point here is that $u \in BV$, because the function $\max\{\cdot,\cdot\}$ is Lipschitz continuous (see e.g. \cite{ambrosio90}); however, computing its distributional derivative in terms of $u_1$ and $u_2$ is a somewhat delicate matter. The strategy of the proof will be to use smooth approximations of $u_1$ and $u_2$, the maximum of which approaches $u$ in at least an $L^1$ sense. Then it will be easy to see that $u$ satisfies the desired inequality.

Let ${\bf \tilde{v}}$ be any continuous vector field with ${\bf \tilde{v}}(t,x) \in c(x,A)$. For $i = 1,2$ we have
\begin{equation}
-\partial_t u_i - {\bf \tilde{v}} \cdot Du_i \leq f
\end{equation}
in distribution.
Let $\xi_\epsilon \geq 0$ be the standard convolution kernel used throughout, let $u_i^\epsilon := \xi_\epsilon \ast u_i$ and $f^\epsilon := \xi_\epsilon \ast f$. For $\phi \geq 0$ a continuous function, we have
\begin{equation} \label{eq:uniqueness_convolution}
-\int_0^T \int_{\bb{T}^N} \phi(t,x)(\partial_t u_i^\epsilon(t,x) + {\bf \tilde{v}}(t,x) \cdot Du_i^\epsilon(t,x)) dt dx \leq \int_0^T \int_{\bb{T}^N} f^\epsilon(t,x)\phi(t,x)dt dx + I_\epsilon,
\end{equation}
where
\begin{equation}
I_\epsilon := \int_0^T \int_{\bb{T}^N} \int_0^T \int_{\bb{T}^N} \xi_\epsilon(t-s,x-y) ({\bf \tilde{v}}(t,x) - {\bf \tilde{v}}(s,y)) Du_i(s,y)\phi(t,x) dt dx.
\end{equation}
Recalling that $\xi_\epsilon$ is supported in $B_\epsilon(0)$ and $\int \int \xi_\epsilon = 1$, an application of Fubini shows that
\begin{equation}
|I_\epsilon| \leq \|\phi\|_\infty \omega(\epsilon) \|Du_i\|,
\end{equation}
where $\omega(\cdot)$ is the modulus of continuity of ${\bf \tilde{v}}$. In particular, $I_\epsilon \to 0$ as $\epsilon \to 0$.

By a simple density argument, (\ref{eq:uniqueness_convolution}) holds for $\phi \geq 0, \phi \in L^\infty$. In particular, one has
\begin{align}
\label{eq:uniqueness_1}
-\iint_{u_1^\epsilon > u_2^\epsilon} \phi(t,x)(\partial_t u_1^\epsilon(t,x) + {\bf \tilde{v}}(t,x) \cdot Du_1^\epsilon(t,x)) dt dx &\leq \iint_{u_1^\epsilon > u_2^\epsilon} f^\epsilon(t,x)\phi(t,x)dt dx + o(1),\\
\label{eq:uniqueness_2}
-\iint_{u_1^\epsilon \leq u_2^\epsilon} \phi(t,x)(\partial_t u_2^\epsilon(t,x) + {\bf \tilde{v}}(t,x) \cdot Du_2^\epsilon(t,x)) dt dx &\leq \iint_{u_1^\epsilon \leq u_2^\epsilon} f^\epsilon(t,x)\phi(t,x)dt dx + o(1)
\end{align} 
as $\epsilon \to 0$.

Let $u^\epsilon := \max\{u_1^\epsilon,u_2^\epsilon\}$. Then
\begin{equation}
(\partial_t u^\epsilon, Du^\epsilon) = \left\{ \begin{array}{ccc}
(\partial_t u_1^\epsilon, Du_1^\epsilon) & \text{on} & \{u_1^\epsilon > u_2^\epsilon\}\\
(\partial_t u_2^\epsilon, Du_2^\epsilon) & \text{on} & \{u_2^\epsilon > u_1^\epsilon\}\\
(\partial_t u_1^\epsilon, Du_1^\epsilon) = (\partial_t u_2^\epsilon, Du_2^\epsilon) & \text{in} & \mathrm{Int}(\{u_1^\epsilon = u_2^\epsilon\}).
\end{array} \right.
\end{equation}
So adding (\ref{eq:uniqueness_1}) and (\ref{eq:uniqueness_2}) yields
\begin{equation} \label{eq:uniqueness_3}
-\int_0^T \int_{\bb{T}^N} \phi(t,x)(\partial_t u^\epsilon(t,x) + {\bf \tilde{v}}(t,x) \cdot Du^\epsilon(t,x)) dt dx \leq \int_0^T \int_{\bb{T}^N} f^\epsilon(t,x)\phi(t,x)dt dx + o(1),
\end{equation}
for any bounded function $\phi \geq 0$. In particular,
\begin{equation} \label{eq:uniqueness_3}
\int_0^T \int_{\bb{T}^N} u^\epsilon(t,x)(\partial_t \phi(t,x) + \mathrm{div}({\bf \tilde{v}}(t,x) \phi(t,x))) dt dx \leq \int_0^T \int_{\bb{T}^N} f^\epsilon(t,x)\phi(t,x)dt dx + o(1),
\end{equation}
for any $\phi \geq 0$ smooth with compact support in $(0,T) \times \bb{T}^N$. Let $\epsilon \to 0$. Since $u^\epsilon \to u$ in $L^1$, it follows that
\begin{equation} \label{eq:uniqueness_4}
\int_0^T \int_{\bb{T}^N} u(t,x)(\partial_t \phi(t,x) + \mathrm{div}({\bf \tilde{v}}(t,x) \phi(t,x))) dt dx \leq \int_0^T \int_{\bb{T}^N} f(t,x)\phi(t,x)dt dx
\end{equation}
for $\phi \geq 0$ smooth with compact support in $(0,T) \times \bb{T}^N$. This can be rewritten as
\begin{equation} \label{eq:uniqueness_5}
-\int_0^T \int_{\bb{T}^N} \phi(t,x)(\partial_t u(t,x) + {\bf \tilde{v}}(t,x) \cdot Du(t,x)) \leq \int_0^T \int_{\bb{T}^N} f(t,x)\phi(t,x)dt dx,
\end{equation}
which holds by density for all continuous functions $\phi \geq 0$. Equation (\ref{eq:uniqueness_5}) is precisely (\ref{eq:inequality_in_measure}), which is what we needed to show.

This completes the proof.

\end{proof}

\begin{proof}[Proof of Corollary \ref{cor:approximation}]
Suppose $(u,f) \in \tilde{\s{K}}$ is a minimizer of Problem \ref{pr:new_relaxed}. Let $f_n$ be a sequence of minimizers of Problem \ref{pr:min_f}. By Propositions \ref{prop:original} and \ref{prop:relaxed} we know that $\s{J}(f_n) \to \s{A}(u,f)$. Moreover, by the proof of Theorem \ref{thm:existence} we see that without loss of generality, $f_n$ is non-negative and Lipschitz. Let $u_n$ be the unique viscosity solution to (\ref{eq:back_hj}) with $f$ replaced by $f_n$. Again, by the proof of Theorem \ref{thm:existence} we have that $u_n \to \tilde{u}$ almost everywhere and $f_n \rightharpoonup \tilde{f}$ weakly in $L^p$ for a pair $(\tilde{u},\tilde{f}) \in \tilde{\s{K}}$ which is a minimizer for Problem \ref{pr:new_relaxed}. By Theorem \ref{thm:minimizers_weak}, $u = \tilde{u}$ on $\{m > 0\}$ and $f = k(\cdot,\cdot,m) = \tilde{f}$, where $m$ is the unique $L^q$ function such that $(m,{\bf w}) \in \s{K}_1$ is a minimizer of Problem \ref{pr:dual} for some $\bf w$. This completes the proof.
\end{proof}

\section{Applications}

\label{sec:conclusion}

To conclude this article, we examine some possible applications of the main result.

{\em Front propagation.} Let us return to the motivating application. Recall that $\tilde{\Omega}_t \subset \bb{R}^N$ is supposed to be an evolving set given by
\begin{equation} \label{eq:moving_front}
\tilde{\Omega}_t = \{y(t) : u_T(y(T)) + \int_t^T f(s,y(s)) \leq 0, ~ y ~\text{solves}~ (\ref{eq:open_loop}), ~ \alpha \in L^\infty(t,T;A), y(T) \in \Omega_T\}
\end{equation}
which is precisely the sub-level set $\{u(t,x) \leq 0\}$ where $u$ is given by (\ref{eq:value}). There are several interpretations of this one might consider. As a first example, assume that $u_T \geq 0$ and that $f \geq 0$, so that $\Omega_T$ is precisely the level set where $u_T = 0$ and in formula (\ref{eq:moving_front}) the inequality can be replaced by equality. Define an evolving obstacle $K_\tau, \tau \in [0,T]$ by
\begin{equation}
K_\tau := \{x : f(\tau,x) > 0\}.
\end{equation}
Assuming $f$ is continuous, it follows that $K_\tau$ is an evolving open set in space. Moreover, we have the following characterization \cite{altarovici12,bokanowski10}:
\begin{equation}
\tilde{\Omega}_t = \{y(t) : y ~\text{solves}~ (\ref{eq:open_loop}), ~ \alpha \in L^\infty(t,T;A), y(T) \in \Omega_T, y(\tau) \notin K_\tau ~ \forall ~ \tau \in [t,T]\}.
\end{equation}
Compare with the ``dynamic blocking" strategy characterized by Equation (\ref{eq:dynamic_blocking}), cf. \cite{bressan07}. Here $K_\tau$ represents an evolving obstacle which in this case completely blocks the progress of the front. Note that $K_\tau$ is not a curve, but a region in space, and moreover we assume that it is possible (if desired) to remove parts of the obstacle once it has been constructed.

We consider Problem \ref{pr:min_f} with $m_0$ given by, say,
\begin{equation}
m_0(x) = \left\{\begin{array}{ccc} 1 & \text{if} & x \in \s{R} \\ 0 & \text{if} & x \notin \s{R} \end{array} \right.
\end{equation}
where $\s{R}$ is a given region in space. Then the quantity $\s{J}(f)$ to be minimized is simply
\begin{equation}
\s{J}(f) = \int_0^T \int_{\bb{T}^N} K(t,x,f(t,x)) dx dt - \int_{\s{R}} u(0,x) dx.
\end{equation}
Note that $\int_{\s{R}} u(0,x) dx$ is essentially an expected value of $u(0,x)$ on the region $\s{R}$. If it is positive, then we can expect on average that $x \in \s{R}$ is not contained in $\tilde{\Omega}_0 = \{u(0,\cdot) \leq 0\}$. The cost $K$ with respect to $f$ models in some sense the cost of constructing the obstacle.

We summarize the implications of Theorem \ref{thm:minimizers_weak}. Assume $K$ satisfies the hypotheses prescribed in Section \ref{sec:assumptions}. Then the infimum of $\s{J}(f)$ is achieved by a unique function $f \in L^p$ and a corresponding $u \in BV$ for which there is uniqueness almost everywhere on the set $\{f > 0\}$, which is equal to the set $\{m > 0\}$ where $m = m(t,x)$ is the evolution of $m_0(x)$ corresponding to an {\em optimal transport} of the set $\s{R}$ (in the sense that $m$ solves the dual problem \ref{pr:dual}), where the vector field causing the transport corresponds in a very weak sense to the normalized gradient of $u$.

There are a few weaknesses to this model of front propagation with obstacles. In the first place, we do not obtain a point-wise result on $u(0,\cdot)$, so it remains theoretically possible that much of the dangerous region is still reached by the front. At least we are guaranteed that not all of it is reached. In the second place, the cost of constructing an obstacle is modeled here by a cost which is bounded below and above by the $L^p$ norm of $f$, whereas the obstacle itself is said to correspond only to the set $\{f > 0\}$. This means that in some sense it costs more to build up the obstacle higher at a given point, even though any positive height does the exact same thing to the front (namely stop the front). A more natural cost functional might be given by a cost distribution over the area where the obstacle is constructed, i.e. a functional of the form
\begin{equation}
\int_0^T \int_{K_\tau} \s{C}(t,x)dx dt.
\end{equation}
Such a cost functional would not satisfy the hypotheses of Section \ref{sec:assumptions}.

Let us consider another application.

{\em Mean field games.} Consider the value function $u$ given by (\ref{eq:value}). In terms of game theory, $u(0,x)$ can be thought of as the minimum price paid by a player who starts at position $x$, can control his velocity by selecting $\alpha(t) \in A$ at time $t$, and must pay a final cost determined by $u_T$ and a running cost determined by $f$. Suppose the initial (continuum) distribution of players is given by $m_0(x)$. For any average player seeking to maximize his payoff, or equivalently minimize his cost, the optimal expected payoff will be given by $-\int u(0,x)m_0(x)dx$. On the other hand, $\int \int K(t,x,f)$ can be thought of as the average cost to the other players, or by symmetry the average benefit to a given player, of imposing a running cost on a player's strategy. If the game is at a Nash equilibrium, then no single player should gain an advantage in changing strategy. Heuristically, this means both the average payoff on the one hand and the average external benefit as a result of competition should be together minimized. Thus, the minimization problem \ref{pr:min_f} corresponds to Nash equilibrium. If $m(t,\cdot)$ is the Nash equilibrium distribution of players at time $t$, then it evolves according to the continuity equation (\ref{eq:continuity_weak}) where $\bf v$ corresponds heuristically to the optimal control $\bar{\alpha}$, i.e. ${\bf v}(t,x) = c(x,\bar{\alpha}(t))$.

It would be interesting to apply the results to differential games with finitely many players. Here we give a sketch of such an application. Define a cost functional on a family of controls $\alpha_1,\ldots,\alpha_n$, corresponding to $n$ players, given by
\begin{equation}
J_i^n(\alpha_1,\ldots,\alpha_n) = \int_0^T \tilde{k}_n\left(s,y_x^{\alpha_i}(s),\frac{1}{N-1} \sum_{j\neq i} \delta_{y_x^{\alpha_i}(s)} \right) ds + u_T(y_x^{\alpha_i}(T)),
\end{equation}
where $y_x^{\alpha_i}$ solves (\ref{eq:open_loop1}) with $t = 0$ and control $\alpha_i$. Here $\tilde{k}_n$ is some smooth approximation of $k$. Suppose now that player $i$ starts at position $x_0^i$ which is a $\bb{T}^N$-valued random variable on a probability space $\Omega$. Each player then has a strategy $\beta^i : \Omega \times \bb{T}^N \times \s{A}^{n-1} \to \s{A}$ where $\s{A} := L^\infty(0,T;A)$ is the space of admissible controls; if $\alpha_i : \Omega^n \times (\bb{T}^N)^n \to \s{A}$, $i=1,\ldots,n$ are the controls (randomly) chosen by players 1 through $n$, then the strategies are determined by
\begin{equation}
\beta^i(\omega^i,x_0^i,(\alpha^j(\omega,x_0))_{j\neq i}) = \alpha^i(\omega,x_0)
\end{equation}
where $x_0 = (x_0^1,\ldots,x_0^n)$ and $\omega = (\omega^1,\ldots,\omega^n)$. Then the cost to a player $i$ of a family of strategies $(\beta^1,\ldots,\beta^n)$ is given by
\begin{equation}
{\bf J}_i^n(\beta^1,\ldots,\beta^n) = \int_{\Omega^n \times (\bb{T}^N)^n} J_i^n(\alpha^1(\omega,x),\ldots,\alpha^n(\omega,x))\Pi_{j=1}^N\bb{P}(d\omega_j)m_p(dx_0^j).
\end{equation}
A {\em Nash equilibrium} for this differential game is a strategy  $(\bar{\beta}^1,\ldots,\bar{\beta}^n)$ such that
\begin{equation}
{\bf J}_i^n(\beta^i,(\bar{\beta}^j)_{j\neq i}) \geq {\bf J}_i^n(\bar{\beta}^1,\ldots,\bar{\beta}^n)
\end{equation}
for any strategy $\beta^i$, for any $i = 1,\ldots,n$. We conjecture that an {\em approximate} Nash equilibrium is given by an {\em open loop} strategy (i.e. a strategy depending only on $\omega$ and $x_0$) as follows. Let $(m,m{\bf v})$ be a minimizer of Problem \ref{pr:dual}. Then $m$ induces a measure $\eta$ on continuous trajectories in $\Gamma := C([0,T];\bb{T}^N)$ such that $e_t \# \eta = m(t)$, where $e_t : \Gamma \to \bb{T}^N$ is the evaluation map $e_t(\gamma) = \gamma(t)$. By disintegration there exists a Borel measurable family of probabilities $(\eta_x)_{x \in \bb{T}^N}$ on $\Gamma$ such that $\eta(d\gamma) = \int_{\bb{T}^N}\eta_x(d\gamma)m_0(dx)$, where for $m_0$-a.e. $x \in \bb{T}^N$ we hae that $\eta_x$-almost every trajectory $\gamma$ starts at $x$. Using the Blackwell-Dubins Theorem \cite{blackwell83}, we can represent $(\eta_x)_{x \in \bb{T}^N}$ by a single measurable map $\bar{\beta}:\Omega \times \bb{T}^N \to \Gamma$ satisfying
\begin{equation}
\int_\Gamma G(\gamma)d\eta(\gamma) = \int_{\Omega \times \bb{T}^N} G(\bar{\beta}(\omega,x))d\bb{P}(\omega)m_0(dx).
\end{equation}
\begin{conjecture}
Under certain assumptions on $k$, the open loop strategy $\bar{\beta}$ is an approximate Nash equilibrium for the differential game in the sense that, for $\epsilon > 0$, we have for $n$ large enough
\begin{equation}
{\bf J}_i^n(\beta^i,(\bar{\beta})_{j\neq i}) \geq {\bf J}_i^n(\bar{\beta},(\bar{\beta})_{j\neq i}) - \epsilon
\end{equation}
for all strategies $\beta^i$ of player $i$, for any $i = 1,\ldots,n$.
\end{conjecture}
One might hope to prove such a conjecture based on the approach of Cardaliaguet in \cite{cardaliaguet13}. However, at one key point in the proof one uses the convolution smoothing technique to obtain smooth solutions to the Hamilton-Jacobi equations as approximations of weak solutions. This requires precisely that assumption on $H(x,Du)$ which we have tried to circumvent in this study. Therefore, the question of formulating a different proof remains open.

{\bf Acknowledgement.} The author wishes to thank Hasnaa Zidani (Commands team, ENSTA ParisTech and Inria Saclay) for her support as a mentor during his post-doc at Inria Saclay, for introducing him to the problem, and for many fruitful discussions. He would also like to thank Pierre Cardaliaguet (Paris Dauphine) for many helpful discussions and insights.

\bibliographystyle{amsalpha}
\bibliography{optimal_control_HJ_eqns}

\newcommand{\etalchar}[1]{$^{#1}$}
\providecommand{\bysame}{\leavevmode\hbox to3em{\hrulefill}\thinspace}
\providecommand{\MR}{\relax\ifhmode\unskip\space\fi MR }
\providecommand{\MRhref}[2]{%
  \href{http://www.ams.org/mathscinet-getitem?mr=#1}{#2}
}
\providecommand{\href}[2]{#2}
\begin{thebibliography}{ABZ{\etalchar{+}}12}

\bibitem[ABZ{\etalchar{+}}12]{altarovici12}
Albert Altarovici, Olivier Bokanowski, Hasnaa Zidani, et~al., \emph{A general
  {Hamilton-Jacobi} framework for nonlinear state-constrained control
  problems}, ESAIM: Control, Optimisation and Calculus of Variations (2012).

\bibitem[AC08]{ambrosio08}
Luigi Ambrosio and Gianluca Crippa, \emph{Existence, uniqueness, stability and
  differentiability properties of the flow associated to weakly differentiable
  vector fields}, Transport equations and multi-D hyperbolic conservation laws,
  Springer, 2008, pp.~3--57.

\bibitem[ADM90]{ambrosio90}
Luigi Ambrosio and Gianni Dal~Maso, \emph{A general chain rule for
  distributional derivatives}, Proceedings of the American Mathematical Society
  \textbf{108} (1990), no.~3, 691--702.

\bibitem[BBFJ08]{bressan08}
Alberto Bressan, Maria Burago, Arthur Friend, and Jessica Jou, \emph{Blocking
  strategies for a fire control problem}, Analysis and Applications \textbf{06}
  (2008), no.~03, 229--246.

\bibitem[BCD97]{bardi97}
Martino Bardi and Italo Capuzzo-Dolcetta, \emph{Optimal control and viscosity
  solutions of {Hamilton-Jacobi-Bellman} equations}, Birkh\"{a}user, 1997.

\bibitem[BD83]{blackwell83}
David Blackwell and Lester~E Dubins, \emph{An extension of {Skorohod’s}
  almost sure representation theorem}, Proceedings of the American Mathematical
  Society \textbf{89} (1983), no.~4, 691--692.

\bibitem[BDL09]{bressan09b}
Alberto Bressan and Camillo De~Lellis, \emph{Existence of optimal strategies
  for a fire confinement problem}, Communications on Pure and Applied
  Mathematics \textbf{62} (2009), no.~6, 789--830.

\bibitem[BFZ10]{bokanowski10}
Olivier Bokanowski, Nicolas Forcadel, and Hasnaa Zidani, \emph{Reachability and
  minimal times for state constrained nonlinear problems without any
  controllability assumption}, SIAM Journal on Control and Optimization
  \textbf{48} (2010), no.~7, 4292--4316.

\bibitem[Bre07]{bressan07}
Alberto Bressan, \emph{Differential inclusions and the control of forest
  fires}, Journal of Differential Equations \textbf{243} (2007), no.~2, 179 --
  207, Special Issue in Honor of Arrigo Cellina and Jim Yorke.

\bibitem[Bre13]{bressan13}
Alberto Bressan, \emph{Dynamic blocking problems for a model of fire
  propagation}, Advances in Applied Mathematics, Modeling, and Computational
  Science (Roderick Melnik and Ilias~S. Kotsireas, eds.), Fields Institute
  Communications, vol.~66, Springer US, 2013, pp.~11--40 (English).

\bibitem[BW09]{bressan09}
Alberto Bressan and Tao Wang, \emph{The minimum speed for a blocking problem on
  the half plane}, Journal of Mathematical Analysis and Applications
  \textbf{356} (2009), no.~1, 133 -- 144.

\bibitem[BW10]{bressan09a}
\bysame, \emph{Equivalent formulation and numerical analysis of a fire
  confinement problem}, ESAIM: Control, Optimisation and Calculus of Variations
  \textbf{16} (2010), 974--1001.

\bibitem[Car13]{cardaliaguet13}
Pierre Cardaliaguet, \emph{Weak solutions for first order mean field games with
  local coupling}, arXiv preprint arXiv:1305.7015 (2013).

\bibitem[CCN12]{cardaliaguet12a}
Pierre Cardaliaguet, Guillaume Carlier, and Bruno Nazaret, \emph{Geodesics for
  a class of distances in the space of probability measures}, Calculus of
  Variations and Partial Differential Equations (2012), 1--26.

\bibitem[CS12]{cardaliaguet12}
Pierre Cardaliaguet and Luis Silvestre, \emph{H{\"o}lder continuity to
  {Hamilton-Jacobi} equations with superquadratic growth in the gradient and
  unbounded right-hand side}, Communications in Partial Differential Equations
  \textbf{37} (2012), no.~9, 1668--1688.

\bibitem[DLR11]{delellis11}
Camillo De~Lellis and Roger Robyr, \emph{{Hamilton Jacobi} equations with
  obstacles}, Archive for Rational Mechanics and Analysis \textbf{200} (2011),
  1051--1073 (English).

\bibitem[EG92]{evans92}
Lawrence~C Evans and Ronald~F Gariepy, \emph{Measure theory and fine properties
  of functions}, vol.~5, CRC press, 1992.

\bibitem[GW73]{greene73}
Robert~E Greene and Hung Wu, \emph{On the subharmonicity and
  plurisubharmonicity of geodesically convex functions}, Indiana Univ. Math. J
  \textbf{22} (1973), no.~7, 641--653.

\bibitem[GW74]{greene74}
RE~Greene and H~Wu, \emph{Integrals of subharmonic functions on manifolds of
  nonnegative curvature}, Inventiones mathematicae \textbf{27} (1974), no.~4,
  265--298.

\bibitem[Ish95]{ishii95}
Hitoshi Ishii, \emph{On the equivalence of two notions of weak solutions,
  viscosity solutions and distribution solutions}, Funkcial. Ekvac \textbf{38}
  (1995), no.~1, 101--120.

\bibitem[JLM01]{juutinen01}
Petri Juutinen, Peter Lindqvist, and Juan~J Manfredi, \emph{On the equivalence
  of viscosity solutions and weak solutions for a quasi-linear equation}, SIAM
  journal on mathematical analysis \textbf{33} (2001), no.~3, 699--717.

\bibitem[LL06a]{lasry06}
Jean-Michel Lasry and Pierre-Louis Lions, \emph{Jeux {\`a} champ moyen. i--le
  cas stationnaire}, Comptes Rendus Math{\'e}matique \textbf{343} (2006),
  no.~9, 619--625.

\bibitem[LL06b]{lasry06a}
\bysame, \emph{Jeux {\`a} champ moyen. ii--horizon fini et contr{\^o}le
  optimal}, Comptes Rendus Math{\'e}matique \textbf{343} (2006), no.~10,
  679--684.

\bibitem[LL07]{lasry07}
\bysame, \emph{Mean field games}, Japanese Journal of Mathematics \textbf{2}
  (2007), no.~1, 229--260.

\bibitem[Res68]{reshetnyak68}
Yu~G Reshetnyak, \emph{Weak convergence of completely additive vector functions
  on a set}, Siberian Mathematical Journal \textbf{9} (1968), no.~6,
  1039--1045.

\end{thebibliography}

\end{document}